\documentclass{amsproc}
\usepackage{graphicx,mathrsfs}

\newtheorem{theorem}{Theorem}[section]
\newtheorem{lemma}[theorem]{Lemma}
\newtheorem{corollary}[theorem]{Corollary}
\newtheorem{proposition}[theorem]{Proposition}
\newtheorem{remark}[theorem]{Remark}
\newtheorem{definition}[theorem]{Definition}

\theoremstyle{remark}

\numberwithin{equation}{section}

\newcommand{\cz}{{\mathbb C}}

\newcommand{\nz}{{\mathbb N}}

\newcommand{\rz}{{\mathbb R}}

\newcommand{\calA}{\mathcal{A}}
\newcommand{\calC}{\mathcal{C}}
\newcommand{\calD}{\mathcal{D}}

\newcommand{\calH}{\mathcal{H}}

\newcommand{\calL}{\mathcal{L}}
\newcommand{\calS}{\mathcal{S}}

\newcommand{\scrC}{\mathscr{C}}

\newcommand{\scrL}{\mathscr{L}}

\newcommand{\dbar}{d\hspace*{-0.08em}\bar{}\hspace*{0.1em}}

\newcommand{\eps}{\varepsilon}

\newcommand{\forget}[1]{}

\newcommand{\lra}{\longrightarrow}

\newcommand{\re}{\text{\rm Re}\,}

\newcommand{\smsum}{\mathop{\mbox{\large$\sum$}}}
\newcommand{\spk}[1]{\langle#1\rangle}
\newcommand{\st}{\mbox{\boldmath$\;|\;$\unboldmath}}

\newcommand{\wt}{\widetilde}

\begin{document}

\title[$H_\infty$-calculus for Douglis-Nirenberg systems of mild regularity]%
{Bounded $H_\infty$-calculus for pseudodifferential\\ Douglis-Nirenberg systems of 
mild regularity}

\author{R.\ Denk}
\author{J.\ Saal}
\address{Universit\"at Konstanz, Fachbereich f\"ur Mathematik und Statistik,
         78457 Konstanz, Germany}
\email{robert.denk@uni-konstanz.de, juergen.saal@uni-konstanz.de}
\author{J.\ Seiler}
\address{Universit\"at Hannover, Institut f\"ur Angewandte Mathematik,
         Welfengarten 1, 30167 Hannover, Germany}
\email{seiler@ifam.uni-hannover.de}

\begin{abstract}
We consider pseudodifferential Douglis-Nirenberg systems on $\rz^n$ with components belonging to the 
standard H\"ormander class $S^*_{1,\delta}(\rz^n\times\rz^n)$, $0\le\delta<1$. Parameter-ellipticity 
with respect to a subsector $\Lambda\subset\cz$ is introduced and shown to imply the existence of a 
bounded $H_\infty$-calculus in suitable scales of Sobolev, Besov, and H\"older spaces. We also admit non 
pseudodifferential perturbations. Applications concern systems with coefficients of mild H\"older 
regularity and the generalized thermoelastic plate equations. 
\end{abstract}

\maketitle

\tableofcontents
\section{Introduction}\label{sec:dn01}

The concept of maximal regularity is an important tool in the modern analysis of nonlinear (parabolic) evolution equations. For a densely defined closed operator $A:\calD(A)\subset X\to X$ in a Banach space $X$, maximal $L_q$-regularity essentially means that the initial value problem $u_t+Au(t)=f(t)$, $u(0)=0$, for each right-hand side $f\in L_q(\rz_+,X)$ admits a unique solution with $Au\in L_q(\rz_+,X)$ (in case of invertibility of $A$ this is equivalent to $u\in W^1_q(\rz_+,X)\cap L_q(\rz_+,\calD(A))$). In combination with fixed point arguments maximal $L_q$-regularity may be used to deduce existence and regularity results for solutions of nonlinear problems. 

It is known that $A$ is the generator of an analytic semi-group in $X$, provided it has maximal regularity. The reverse implication, however, is false. Thus it is natural to address the  question, which conditions on $A$ imply maximal regularity. One such condition is the existence of a so-called bounded $H_\infty$-calculus for $A$. This is a functional calculus, that allows to define $f(A)\in\scrL(X)$ for certain complex-valued holomorphic functions $f$; for a short review see Subsection \ref{sec:dn04.2}. This calculus was introduced by McIntosh in \cite{Mcin} and recieved since then a lot of attention (cf. \cite{DDHPV} and \cite{KuWe} for extensive expositions and further literature). The existence of an $H_\infty$-calculus implies existence of bounded imaginary powers. Combining this with a classical result of Dore and Venni \cite{DoVe}, maximal regularity follows. An alternative approach to maximal regularity relies on the so-called $\mathcal{R}$-boundedness of the resolvent. 

The aim of the present paper is to establish conditions for perturbed Douglis-Nirenberg systems that ensure the existence of a bounded $H_\infty$-calculus, hence of maximal regularity. We consider pseudodifferential systems on $\rz^n$ with components whose symbols belong to the standard H\"ormander class $S^*_{1,\delta}(\rz^n\times\rz^n)$, $0\le\delta<1$ (the order $*$ is different for each component). The established condition is a condition of parameter-ellipticity with respect to a sector $\Lambda\subset\cz$ containing the left half-plane, called $\Lambda$-ellipticity throughout the paper. We give two, initially seemingly different, formulations of $\Lambda$-ellipticity (see Definitions \ref{dn02.4} and \ref{dn02.5}). The first is motivated by a notion of parameter-ellipticity introduced by Denk, Menniken, and Volevich in \cite{DMV}, which is connected with the so-called Newton-polygon associated with the system. The second formulation is modeled on a condition introduced by Kozhevnikov \cite{Kozh1}, \cite{Kozh2}, for classical (i.e. polyhomogeneous) Douglis-Nirenberg systems. Although different in appearance, we proof that both notions of ellipticity are equivalent. For $\Lambda$-elliptic systems we construct in Section \ref{sec:dn03} a parametrix and show that such systems are diagonalizable modulo smoothing remainders. In Section \ref{sec:dn04} we establish the existence of a bounded $H_\infty$-calculus. 

The perturbations we admit in our analysis allow us not only to consider systems with smooth symbols but also with symbols of a mild H\"older regularity, see Section \ref{sec:dn05}. Minimal regularity assumptions on the symbols (i.e., the coefficients in case of differential systems) are of particular importance when aiming at nonlinear problems. As a further application, see Section \ref{sec:dn06}, we consider the so-called generalized thermoelastic plate equations introduced in \cite{AmBe}, \cite{MuRa}. It has been shown in \cite{DeRa} that (for the involved parameters belonging to the `parabolic region') this equation can be seen as an evolution equation with a generator of an analytic semi-group. We improve this result, showing the existence of a bounded $H_\infty$-calculus. 

\section{Douglis-Nirenberg systems: Basic definitions and properties}\label{sec:dn02}

In this section we provide the basic notation and definitions that will be used throughout 
the paper. Moreover, we recall some standard properties of pseudodifferential operators. 

\begin{definition}\label{dn02.1}
The symbol class $S^\mu_{\delta}(\rz^n\times\rz^n)$ with $\mu\in\rz$ and 
$0\le\delta<1$ consists of all smooth functions $a=a(x,\xi):\rz^n\times\rz^n\to\cz$ satisfying 
 $$\|a\|^\mu_{\delta,k}:=\sup_{\substack{x,\xi\in\rz^n\\|\alpha|+|\beta|\le k}}
   |D^\alpha_\xi D^\beta_x a(x,\xi)|
   \spk{\xi}^{-\mu+|\alpha|-\delta|\beta|}<\infty$$
for any $k\in\nz_0$. As usual, we use the notation $\spk{\xi}:=(1+|\xi|^2)^{1/2}$ 
and $D:=-i\partial$. Frequently, we shall simply write $S^\mu_{\delta}$. 
In case $\delta=0$, we suppress $\delta$ from the notation.
\end{definition}

The system of norms $\|\cdot\|^\mu_{\delta,k}$, $k\in\nz_0$, defines a Fr\'echet topology on 
$S^\mu_{\delta}$. To a given symbol $a\in S^\mu_{\delta}$ we associate a continuous operator 
$a(x,D):\calS(\rz^n)\to\calS(\rz^n)$ by 
\begin{equation*}
 [a(x,D)u](x)=\int_{\rz^n}e^{-ix\xi}a(x,\xi)\widehat{u}(\xi)\,\dbar\xi,
\end{equation*}
where $\widehat{u}$ is the Fourier transform of $u$ and $\dbar\xi=(2\pi)^{-n}d\xi$. 
By duality, we extend this operator to $a(x,D):\calS^\prime(\rz^n)\to\calS^\prime(\rz^n)$. 
This operator restricts to Sobolev spaces in the following way: 

\begin{theorem}\label{dn02.2}
Let $a\in S^{\mu}_{\delta}(\rz^n\times\rz^n)$ and $1<p<\infty$. Then $a(x,D)$ restricts to a 
continuous map  
 $$a(x,D):H^s_p(\rz^n)\longrightarrow H^{s-\mu}_p(\rz^n)$$
for any real $s$. Moreover, we have continuity of the mappings 
\begin{equation}\label{dn02.A}
 a\mapsto a(x,D):S^{\mu}_{\delta}(\rz^n\times\rz^n)\longrightarrow
 \calL(H^s_p(\rz^n),H^{s-\mu}_p(\rz^n)).
\end{equation}
\end{theorem}

The continuity of \eqref{dn02.A} entails that the norm $\|a(x,D)\|$ as a bounded operator 
between Sobolev spaces can be estimated from above by $C\,\|a\|^\mu_{\delta,k}$ with 
suitable constants $k$ and $C$ that do not depend on $a$. 

Pseudodifferential operators behave well under composition: There exists a continuous map 
 $$(a_1,a_2)\mapsto a_1\#a_2:S^{\mu_1}_{\delta}\times S^{\mu_2}_{\delta}
   \longrightarrow S^{\mu_1+\mu_2}_{\delta}$$
such that $a_1(x,D)a_2(x,D)=(a_1\#a_2)(x,D)$. For an explicit formula of the so-called 
Leibniz-product $a_1\#a_2$ see, for example \cite{Kuma}. In the sense of an asymptotic 
expansion we have 
 $$a_1\#a_2\sim\sum\limits_{\alpha\in\nz_0^n} 
   \frac{1}{\alpha!}\partial^\alpha_\xi a_1 D^\alpha_x a_2,$$ 
i.e., for any positive integer $N$, 
 $$a_1\#a_2-\sum_{|\alpha|=0}^{N-1}\frac{1}{\alpha!}\partial^\alpha_\xi a_1 D^\alpha_x a_2
   \;\in\;S^{\mu_1+\mu_2-(1-\delta)N}_{\delta}.$$

\begin{definition}\label{dn02.3}
A Douglis-Nirenberg system is a $(q\times q)$-matrix, $q\in\nz$, of pseudodifferential operators 
 $$A(x,D)=\Big(a_{ij}(x,D)\Big)_{1\le i,j\le q}$$
such that there exist real numbers $m_1,\ldots,m_q$ and $l_1,\ldots,l_q$ with the property that 
 $$a_{ij}(x,\xi)\in S^{l_i+m_j}_{\delta}(\rz^n\times\rz^n)\qquad\forall\;i,j=1,\ldots,q$$ 
and the numbers $r_i:=l_i+m_i$ satisfy $r_1\ge r_2\ge\ldots\ge r_q\ge0$. 
\end{definition}

A Douglis-Nirenberg system in the sense of the previous definition induces continuous operators 
 $$A(x,D):\mathop{\mbox{\Large$\oplus$}}_{j=1}^q H^{s+m_j}_p(\rz^n)
   \lra\mathop{\mbox{\Large$\oplus$}}_{i=1}^q H^{s-l_i}_p(\rz^n)
   \qquad \forall\;s\in\rz.$$
Due to the requested nonnegativity of the $r_i$ we have that $s+m_i\ge s-l_i$. 
Therefore, we may (and will) consider $A(x,D)$ as an unbounded operator in 
$\mathop{\mbox{\Large$\oplus$}}\limits_{i=1}^q H^{s-l_i}_p$ with domain 
$\mathop{\mbox{\Large$\oplus$}}\limits_{j=1}^q H^{s+m_j}_p(\rz^n)$.

\section{$\Lambda$-elliptic Douglis-Nirenberg systems}\label{sec:dn03}

\subsection{Parameter-ellipticity}\label{sec:dn02.2}

From now on let $\Lambda$ denote a closed subsector of the complex plain, i.e. 
\begin{equation}\label{dn01.C}
 \Lambda=\Lambda(\theta)=\big\{re^{i\varphi}\st r\ge 0,\;
 \theta\le\varphi\le 2\pi-\theta\big\},\qquad 0<\theta<\pi.
\end{equation}
We let $A(x,D)$ be a system as in Definition {\rm\ref{dn02.3}}. Moreover, for simplicity of exposition, we shall assume from now on that 
\begin{equation}\label{dn02.B}
 r_1> r_2>\ldots>r_q\ge 0,\qquad r_i:=l_i+m_i.
\end{equation}
Let us point out that this assumption is mainly made for notational convenience; the 
main results of the present paper, i.e. parametrix construction, diagonalization, and existence 
of a bounded $H_\infty$-calculus, remain valid (in an adapted formulation) also 
in the general case when in \eqref{dn02.B} some (or all) of the inequalities are replaced by 
equalities. Let us also mention that we assume neither any ordering nor positivity or negativity 
of the numbers $l_1,\ldots,l_q$, $m_1,\ldots,m_q$. 

We shall now introduce two notions of parameter-ellipticity, where the parameter-space is 
just the above sector $\Lambda$, and then show that they are equivalent. 
These conditions are modeled on those given in \cite{DMV} and \cite{Kozh1}, \cite{Kozh2}. To this end let  
\begin{equation}\label{det}
 P(x,\xi;\lambda)=P_A(x,\xi;\lambda):=\mathrm{det}\big(A(x,\xi)-\lambda\big)
\end{equation}
denote the characteristic polynomial of $A(x,\xi)$ $($where we identify $\lambda$ with $\lambda I$ and where 
$I$ denotes the identity matrix$)$. It is straightforward to verify (see also Lemma \ref{dn03.2}) that 
 $$|P(x,\xi;\lambda)|\le C\,(\spk{\xi}^{r_1}+|\lambda|)\cdot\ldots\cdot(\spk{\xi}^{r_q}+|\lambda|)
   \qquad\forall\;x,\xi\in\rz^n\quad\forall\;\lambda\in\cz$$
with a suitable constant $C\ge0$. 

\begin{definition}\label{dn02.4}
$A(x,D)$ is said to be $\Lambda$-elliptic if, for some constants $C>0$ and $R\ge 0$,  
\begin{equation}\label{dn01.D}
|P(x,\xi;\lambda)|\ge C\,(\spk{\xi}^{r_1}+|\lambda|)\cdot\ldots\cdot(\spk{\xi}^{r_q}+|\lambda|)
   \qquad\forall\;x\in\rz^n\quad\forall\;|\xi|\ge R\quad\forall\;\lambda\in\Lambda. 
\end{equation}
\end{definition}

For the second definition let us introduce further  notation. We call  
 $$A[\kappa](x,D)=\Big(a_{ij}(x,D)\Big)_{1\le i,j\le\kappa},\qquad 1\le\kappa\le q,$$
the {\em $\kappa$-th principal minor} of $A(x,D)$ and let  
 $$E_\kappa=\mathrm{diag}(0,\ldots,0,1)\in\cz^{\kappa\times\kappa}.$$ 

\begin{definition}\label{dn02.5}
$A(x,D)$ is called $\Lambda$-elliptic $($with principal minors$)$ if 
\begin{equation}\label{dn01.D.5}
 \big|\mathrm{det}\big(A[\kappa](x,\xi)-\lambda E_\kappa\big)\big|
 \ge C\,\spk{\xi}^{r_1+\ldots+r_{\kappa-1}}(\spk{\xi}^{r_\kappa}+|\lambda|)
 \qquad \forall\;x\in\rz\quad\forall\;|\xi|\ge R\quad\forall\;\lambda\in\Lambda,
\end{equation}
with suitable constants $C>0$ and $R\ge 0$, for each $1\le\kappa\le q$.  
\end{definition}

\begin{theorem}\label{dn02.6}
The two notions of $\Lambda$-ellipticity given in Definition {\rm\ref{dn02.4}} and 
{\rm\ref{dn02.5}}, respectively,  are equivalent. 
\end{theorem}
{\sc Proof}.
That $\Lambda$-elliptic with principal minors implies $\Lambda$-ellipticity in the sense of 
Definition \ref{dn02.4} we shall prove in Corollary \ref{dn03.1.5}, below. For the other implication, 
we proceed in two steps: 

{\bf Step 1:} First we will show that the condition of $\Lambda$-ellipticity with principal minors 
is satisfied for $\lambda=0$. More precisely, we will show that there exist $R\ge0$ and $C>0$ such 
that for all $\kappa=1,\dots,q$ we have
 $$|\det A[\kappa](x,\xi)|\ge C \langle \xi\rangle^{r_1+\dots+r_\kappa}\qquad 
   \forall\;x\in\rz^n\quad\forall\;|\xi|\ge R\quad\forall\;\lambda\in\Lambda.$$
Assume this is not the case. Then there exists a $\kappa\in\{1,\dots,q\}$ and a sequence
$(x_k,\xi_k)_{k\in\nz}\subset \rz^n\times \rz^n$ with $|\xi_k|\to\infty$ and 
\begin{equation}\label{rd01}
  \big| \det A[\kappa](x_k,\xi_k)\big| \; \langle
  \xi_k\rangle^{-r_1-\dots-r_\kappa} \xrightarrow{k\to\infty}0.
\end{equation}
We define $r:= \frac{r_\kappa+r_{\kappa+1}}2\in(r_{\kappa+1},r_\kappa)$ and choose the sequence
$(\lambda_k)_{k\in\nz}\subset\Lambda$ by $\lambda_k := \langle\xi_k\rangle^r \lambda_0$ with a 
fixed $\lambda_0\in\Lambda$, $|\lambda_0|=1$. We will consider the $q\times q$-matrix
 $$\tilde A[\kappa](x,\xi,\lambda) := 
   \begin{pmatrix}
    A[\kappa](x,\xi) & 0 \\ 0 & -\lambda I_{q-\kappa}
   \end{pmatrix}$$
where $I_{q-\kappa}$ stands for the $(q-\kappa)$-dimensional unit matrix. Due to \eqref{rd01} we have
\begin{equation}\label{rd02}
  \frac{|\det \tilde A[\kappa](x_k,\xi_k,\lambda_k)|}{\langle
  \xi_k\rangle^{r_1+\dots+r_\kappa} |\lambda_k|^{q-\kappa}} \xrightarrow{k\to\infty}0.
\end{equation}
For a rescaling of the matrix $\tilde A[\kappa]$, we set $\epsilon_j:= \frac{r-r_j}2 >0$ for 
$j=\kappa+1,\dots,q$ and
\begin{align*}
  D_1(\xi) & := \textrm{diag} \big( \langle \xi\rangle^{-l_1},\dots,
  \langle \xi\rangle^{-l_\kappa},
  \langle\xi\rangle^{-l_{\kappa+1}-\epsilon_{\kappa+1}},\dots,
  \langle\xi\rangle^{-l_q-\epsilon_q}\big),\\
  D_2(\xi) & := \textrm{diag} \big( \langle \xi\rangle^{-m_1},\dots,
  \langle \xi\rangle^{-m_\kappa},
  \langle\xi\rangle^{-m_{\kappa+1}-\epsilon_{\kappa+1}},\dots,
  \langle\xi\rangle^{-m_q-\epsilon_q}\big).
  \end{align*}
We will estimate the coefficients $b_{ij}(x,\xi,\lambda)$ of the matrix
 $$B(x,\xi,\lambda) := D_1(\xi) \Big( (A(x,\xi)-\lambda) - \tilde
   A[\kappa](x,\xi,\lambda)\Big) D_2(\xi).$$ 
For $i,j\le \kappa$ we get
 $$|b_{ij}(x_k,\xi_k,\lambda_k)| = \delta_{ij}
   \langle\xi_k\rangle^{-r_i} |\lambda_k| = \delta_{ij}
   \langle\xi_k\rangle^{r-r_i}\xrightarrow{k\to\infty}0.$$ 
For $i,j>\kappa$ one has
 $$|b_{ij}(x_k,\xi_k,\lambda_k)| = |a_{ij}(x_k,\xi_k)|
   \langle\xi_k\rangle^{-l_i-m_j-\epsilon_i-\epsilon_j}\xrightarrow{k\to\infty}0,$$ 
since $a_{ij}(x,\xi)\in S^{l_i+m_j}(\rz^n\times\rz^n)$ and $\epsilon_i, \epsilon_j>0$. 
In the same way, we get $|b_{ij}|\to 0$ in the cases $i\le \kappa, \,j>\kappa$ and $i>\kappa,\, j\le \kappa$, 
where now only one factor of the form $\langle\xi_k\rangle^{-\epsilon_i}$ appears.
Hence $B(x_k,\xi_k,\lambda_k)\xrightarrow{k\to\infty}0$. 
By direct computation, $D_1(\xi_k)(A(x_k,\xi_k)-\lambda_k)D_1(\xi_k)$ can be shown to be bounded, uniformly in $k$. 
Since the determinant is uniformly continuous on bounded sets, we thus can conclude that  
 $$\det D_1(\xi_k)\big( A(x_k,\xi_k)-\lambda_k\big) D_2(\xi_k) - \det D_1(\xi_k)
   \tilde A[\kappa](x_k,\xi_k,\lambda_k) D_2(\xi_k) \to 0.$$ 
From this, the definition of $|\lambda_k|$, and \eqref{rd02} we obtain
 $$\frac{\big|\det\big(
   A(x_k,\xi_k)-\lambda_k\big)\big|}{\langle\xi_k\rangle^{r_1+\dots+r_\kappa}
   |\lambda_k|^{q-\kappa}}\xrightarrow{k\to\infty}0.$$ 
By our choice of $r$ and $\lambda_k$ we have
 $$\frac{\mathop{\mbox{$\prod$}}\limits_{j=1}^q (\langle
   \xi_k\rangle^{r_j}+|\lambda_k|)}{\langle\xi_k\rangle^{r_1+\dots+r_\kappa}
   |\lambda_k|^{q-\kappa}}\xrightarrow{k\to\infty}1.$$ 
The last two statements yield 
 $$\big|\det \big( A(x_k,\xi_k)-\lambda_k\big)\big| \cdot
   \mathop{\mbox{\large$\prod$}}_{j=1}^q 
   \big( \langle \xi_k\rangle^{r_j} + |\lambda_k|\big)^{-1}\xrightarrow{k\to\infty}0$$ 
which contradicts the $\Lambda$-ellipticity of $A(x,D)$. Thus the conditions of Definition
\ref{dn02.5} are satisfied for $\lambda=0$.

{\bf Step 2:} Now we want to show that condition \eqref{dn01.D.5} holds for $\lambda\not=0$. 
If this is not the case there exists a $\kappa\in\{1,\dots,q\}$ and a sequence
$(x_k,\xi_k,\lambda_k)_{k\in\nz}\subset \rz^n\times\rz^n\times\Lambda$ with $|\xi_k|\to\infty$ and 
\begin{equation}\label{rd03}
  \frac{|\det (A[\kappa](x_k,\xi_k)-\lambda_kE_\kappa)|}{\langle
  \xi_k\rangle^{r_1+\dots+r_{\kappa-1}} (
  \langle\xi_k\rangle^{r_\kappa}+|\lambda_k|)}\xrightarrow{k\to\infty}0.
\end{equation}
We shall use the equality 
  $$\det\big(A[\kappa](x,\xi) - \lambda E_\kappa\big) = \det
    A[\kappa](x,\xi)-\lambda \det A[\kappa-1](x,\xi),$$
which is valid due to the linearity  of the determinant with respect to the $\kappa$-th column. 
\begin{itemize}
\item[(i)] First we show that 
$\liminf_{k\to\infty}\frac{\langle\xi_k\rangle^{r_\kappa}}{|\lambda_k|}>0$. 
If this is not the case we may assume, by passing to a subsequence, that
$\frac{\langle \xi_k\rangle^{r_\kappa}}{|\lambda_k|}\xrightarrow{k\to\infty}0$. 
Now we apply Step 1 of this proof to estimate
\begin{align*}
  & \frac{|\det(A[\kappa](x_k,\xi_k)-\lambda_k E_\kappa)|}{\langle
  \xi_k\rangle^{r_1+\dots+r_{\kappa-1}} (\langle
  \xi_k\rangle^{r_\kappa}+|\lambda_k|)} \\
  &\qquad \ge |\lambda_k|
  \frac{|\det A[\kappa-1](x_k,\xi_k)|}{\langle
  \xi_k\rangle^{r_1+\dots+r_{\kappa-1}} (\langle
  \xi_k\rangle^{r_\kappa}+|\lambda_k|)} - \frac{|\det
  A[\kappa](x_k,\xi_k)|}{\langle
  \xi_k\rangle^{r_1+\dots+r_{\kappa-1}} (\langle
  \xi_k\rangle^{r_\kappa}+|\lambda_k|)}\\
  & \qquad\ge C_1 \; \frac{|\lambda_k|}{\langle
  \xi_k\rangle^{r_\kappa}+|\lambda_k|} - C_2 \; \frac{\langle
  \xi_k\rangle^{r_\kappa}}{\langle
  \xi_k\rangle^{r_\kappa}+|\lambda_k|}
\end{align*}
with two positive constants $C_1$ and $C_2$. For $k\to\infty$ the right-hand side of the last 
inequality tends to $C_1>0$ which contradicts \eqref{rd03}.
\item[(ii)] In the same way we show $\liminf_{k\to\infty}\frac{|\lambda_k|}{\langle\xi_k\rangle^{r_\kappa}}>0$. 
If this does not hold, we may assume
$\frac{|\lambda_k|}{\langle\xi_k\rangle^{r_\kappa}}\xrightarrow{k\to\infty}0$. 
Thus we obtain
 $$\frac{|\det(A[\kappa](x_k,\xi_k)-\lambda_k E_\kappa)|}{\langle
  \xi_k\rangle^{r_1+\dots+r_{\kappa-1}} (\langle
  \xi_k\rangle^{r_\kappa}+|\lambda_k|)} \ge C_1\;\frac{\langle
  \xi_k\rangle^{r_\kappa}}{\langle
  \xi_k\rangle^{r_\kappa}+|\lambda_k|} - C_2\;\frac{|\lambda_k|}{\langle
  \xi_k\rangle^{r_\kappa}+|\lambda_k|}\xrightarrow{k\to\infty}C_1>0,$$
again a contradiction to \eqref{rd03}.
\item[(iii)] Due to (i) and (ii), there exist positive constants $C_3$ and $C_4$ with
 $$C_3 \langle\xi_k\rangle^{r_\kappa} \le |\lambda_k|\le C_4
   \langle\xi_k \rangle^{r_\kappa}$$ 
for sufficiently large $k$. As in Step 1, we use the scaling matrices $D_1(\xi)$ and $D_2(\xi)$, 
now setting $r:= r_\kappa$. For the coefficients of the matrix
  $$B(x,\xi,\lambda) := D_1(\xi) \left[ \big( A(x,\xi)-\lambda I_q\big) - 
    \begin{pmatrix} 
     A[\kappa](x,\xi) - \lambda E_\kappa & 0\\ 
     0 & -\lambda I_{q-\kappa}
     \end{pmatrix}\right] D_2(\xi) $$
we obtain the estimates
 $$\hspace*{5ex}|b_{ij}(x_k,\xi_k,\lambda_k)| =
  \begin{cases}
  \delta_{ij}\langle\xi_k\rangle^{-r_i}|\lambda_k|\quad&:i,j<\kappa\\
  0 \quad&:i=\kappa,j<\kappa\text{ or } i<\kappa, j=\kappa\\
  |a_{ij}(x_k,\xi_k)| \, \langle\xi_k\rangle^{-l_i-m_j-\epsilon_i}\quad&:i>\kappa, j\le \kappa\\
  |a_{ij}(x_k,\xi_k)| \, \langle\xi_k\rangle^{-l_i-m_j-\epsilon_j}\quad&:i\le\kappa,j>\kappa\\
  |a_{ij}(x_k,\xi_k)| \, \langle\xi_k\rangle^{-l_i-m_j-\epsilon_i-\epsilon_j}\quad&:i,j>\kappa
  \end{cases}.$$
In all cases $|b_{ij}(x_k,\xi_k,\lambda_k)|\xrightarrow{k\to\infty}0$. In the same way as before we obtain, 
using \eqref{rd03} and the equality
 $$\frac{\mathop{\mbox{$\prod$}}\limits_{j=1}^q \big( \langle\xi_k\rangle^{r_j} +
   |\lambda_k|\big)}{\langle\xi_k\rangle^{r_1+\dots+r_{\kappa-1}} \big(
   \langle\xi_k\rangle^{r_\kappa} +|\lambda_k|\big)
   |\lambda_k|^{q-\kappa}}\xrightarrow{k\to\infty}1,$$ 
that
 $$\big|\det \big( A(x_k,\xi_k)-\lambda_k\big)\big| \cdot\mathop{\mbox{\large$\prod$}}_{j=1}^q
   \big( \langle\xi_k\rangle^{r_j} + |\lambda_k|\big)^{-1}\xrightarrow{k\to\infty}0.$$ 
This contradicts the $\Lambda$-ellipticity of $A(x,D)$ and finishes the proof.
\hspace*{\fill}\qed 
\end{itemize}

\subsection{Construction of the parametrix}\label{sec:dn03.2}

Throughout this subsection let $A(x,D)$ be a $\Lambda$-elliptic Douglis-Nirenberg system. For simplicity we 
shall assume that \eqref{dn01.D} holds with $R=0$. As the following lemma shows, for our purposes that 
is no restriction: 

\begin{lemma}\label{dn02.7}
Let $A(x,D)$ be $\Lambda$-elliptic. Then there exists an $\alpha_0 \ge0$ such that the 
system $A_\alpha(x,D):=A(x,D)+\alpha$ satisfies 
 $$|P_{A_\alpha}(x,\xi;\lambda)|
   \ge C\,(\spk{\xi}^{r_1}+|\lambda|)\cdot\ldots\cdot(\spk{\xi}^{r_q}+|\lambda|)
   \qquad\forall\;x\in\rz^n\quad\forall\;\xi\in\rz^n\quad\forall\;\lambda\in\Lambda,$$
whenever $\alpha\ge\alpha_0$. 
\end{lemma}
\begin{proof}
By definition, we have 
 $$P_{A_\alpha}(x,\xi;\lambda)=P_A(x,\xi;\lambda-\alpha)
   =\mathrm{det}\big(A(x,\xi)-(\lambda-\alpha)\big).$$
Obviously, there exist constants $d\le 1\le D$ such that 
 $$d\,\spk{\lambda}\le|\lambda-\alpha|\le D\,\spk{\lambda}
   \qquad\forall\;\lambda\in\Lambda.$$
As $\lambda-\alpha\in\Lambda$ for each $\lambda\in\lambda$, the $\Lambda$-ellipticity of 
$A(x,D)$ thus yields that 
 $$|P_{A_\alpha}(x,\xi;\lambda)|
   \ge C_\alpha\,(\spk{\xi}^{r_1}+\spk{\lambda})\cdot\ldots\cdot(\spk{\xi}^{r_q}+\spk{\lambda})$$
uniformly in $x\in\rz^n$, $|\xi|\ge R$ and $\lambda\in\Lambda$. Let us consider those $\xi$ 
with $|\xi|\le R$. Clearly, 
 $$\sup_{x\in\rz^n,\,|\xi|\le R}\|A(x,\xi)\|<\infty.$$
Thus, choosing $\alpha_0$ large enough, $A_\alpha(x,\xi)$ has no spectrum in $\Lambda$ and 
 $$d\,\spk{\lambda}^q\le|P_{A_\alpha}(x,\xi;\lambda)|\le D\,\spk{\lambda}^q 
   \qquad\forall\;\lambda\in\Lambda$$
uniformly in $x\in\rz^n$ and $|\xi|\le R$, for suitable constants $d\le 1\le D$. 
This yields the result. 
\end{proof}

\begin{lemma}\label{dn03.2}
Define  
 $$G^{(0)}(x,\xi;\lambda)=\Big(g_{ij}^{(0)}(x,\xi;\lambda)\Big)_{1\le i,j\le q}
   :=\big(A(x,\xi)-\lambda\big)^{-1}.$$
Then the following uniform in $(x,\xi,\lambda)\in \rz^n\times\rz^n\times\Lambda$ estimates 
hold true: 
 $$|D^\alpha_\xi D^\beta_x g_{ij}^{(0)}(x,\xi;\lambda)|\le C_{\alpha\beta}\,
   (\spk{\xi}^{r_i}+|\lambda|)^{-1}(\spk{\xi}^{r_j}+|\lambda|)^{-1}
   \spk{\xi}^{l_i+m_j-|\alpha|+\delta|\beta|}$$
in case $i\not=j$, and 
 $$|D^\alpha_\xi D^\beta_x g_{ii}^{(0)}(x,\xi;\lambda)|\le C_{\alpha\beta}\,
   (\spk{\xi}^{r_i}+|\lambda|)^{-1}\spk{\xi}^{-|\alpha|+\delta|\beta|}.$$
\end{lemma}
\begin{proof}
According to Cramer's rule we have 
 $$g_{ij}^{(0)}(x,\xi;\lambda)=\frac{1}{P(x,\xi,\lambda)}
   \mathrm{det}(A(x,\xi)-\lambda)^{(i,j)},$$ 
where $B^{(i,j)}$ denotes the matrix obtained by deleting the $j$-th row and $i$-th column 
of the matrix $B$. Let us consider the case $i\not=j$. Set $Z^{(l)}=\{1,\ldots,q\}\setminus\{l\}$. 
Then, suppressing $(x,\xi)$ from the notation, $\mathrm{det}(A(x,\xi)-\lambda)^{(i,j)}$ is 
a linear combination of terms 
 $$(a_{i_1,i_1}-\lambda)\cdots(a_{i_k,i_k}-\lambda)\cdot 
   a_{i_{k+1},\pi i_{k+1}}\cdots a_{i_{q-1},\pi i_{q-1}},$$
where $Z^{(j)}=\{i_1,\ldots,i_{q-1}\}$, $1\le k\le q-2$, and $\pi:Z^{(j)}\to Z^{(i)}$ is 
a bijection. Each of these terms can be estimated from above by 
$\spk{\xi}^{l_i}\spk{\xi}^{m_j}\mathop{\mbox{$\prod$}}\limits_{\substack{l=1\\ l\not=i,j}}
(\spk{\xi}^{r_l}+|\lambda|)$. 
Together with the ellipticity assumption \eqref{dn01.D} this shows the desired estimate 
in case $|\alpha|=|\beta|=0$. The general case follows similarly using chain and product rule. 
The case $i=j$ is analogous. 
\end{proof}

Note also that the estimates of $G^{(0)}$ from the previous lemma for $\alpha=\beta=0$ are 
easily seen to imply the estimate \eqref{dn01.D}. Thus this would yield another equivalent 
definition of $\Lambda$-ellipticity.

As a direct consequence of these estimates, we get the natural fact that $\Lambda$-elliticity is 
preserved under perturbations by lower order terms:

\begin{corollary}\label{dn03.2.5}
Let $A(x,\xi)$ and $\wt{A}(x,\xi)$ be two Douglis-Nirenberg systems  
such that $A(x,\xi)$ is $\Lambda$-elliptic and for each $1\le i,j\le q$ the $(i,j)$-th component of   
$R(x,\xi):={A}(x,\xi)-\wt{A}(x,\xi)$ has order $l_i+m_j-\eps$ for some $\eps>0$. Then also 
$\wt{A}(x,\xi)$ is $\Lambda$-elliptic. 
\end{corollary}
\begin{proof}
For large enough $|\xi|$ we have
 $$\mathrm{det}\big(\wt{A}(x,\xi)-\lambda\big)=\mathrm{det}({A}(x,\xi)-\lambda) 
   \mathrm{det}\big(1+({A}(x,\xi)-\lambda)^{-1}R(x,\xi)\big).$$
Define $M(\xi)=\mathrm{diag}\big(\spk{\xi}^{m_1},\ldots,\spk{\xi}^{m_q}\big)$ and 
$L(\xi)=\mathrm{diag}\big(\spk{\xi}^{l_1},\ldots,\spk{\xi}^{l_q}\big)$. Conjugation with 
$M$ yields 
\begin{equation}\label{eq:ell}
 \mathrm{det}\big(1+({A}(x,\xi)-\lambda)^{-1}R(x,\xi)\big)=
 \mathrm{det}\big(1+M(\xi)G^{(0)}(x,\xi;\lambda)L(\xi)L(\xi)^{-1}R(x,\xi)M(\xi)^{-1}\big).
\end{equation}
The $(i,j)$-th component of $L^{-1}RM^{-1}$ is just 
 $$r_{ij}(x,\xi)\spk{\xi}^{-m_j}\spk{\xi}^{-l_i}\in S^{-\eps}_\delta.$$
Due to Proposition \ref{dn03.2}, the $(i,j)$-th component of $MG^{(0)}L$ can be estimated from 
above by 
\begin{align*}
 |g_{ij}^{(0)}(x,\xi;\lambda)\spk{\xi}^{m_i}\spk{\xi}^{l_j}|\le 
 C(\spk{\xi}^{r_i}+|\lambda|)^{-1}(\spk{\xi}^{r_j}+|\lambda|)^{-1}
 \spk{\xi}^{r_i+r_j}\le C
\end{align*}
for $i\not=j$, and analogously for $i=j$. Therefore the matrix on the right-hand side of \eqref{eq:ell} tends to 
the identity matrix for $|\xi|\to\infty$, uniformly in $(x,\lambda)$. Hence the absolute value 
of the determinant \eqref{eq:ell} can be estimated from below by $1/2$ for sufficiently large 
$|\xi|$ and all $(x,\lambda)\in\rz^n\times\Lambda$. Thus with $A$ also $\wt{A}$ satisfies 
the ellipticity assumption given in Definition \ref{dn02.4}.  
\end{proof}

Proceeding with $G^{(0)}$ from Lemma \ref{dn03.2}, we define recursively for $\nu\in\nz$
\begin{equation}\label{dn01.E}
 G^{(\nu)}(x,\xi;\lambda)=\smsum_{\substack{m+|\alpha|=\nu\\ m<\nu}}\frac{1}{\alpha!}
 (\partial^\alpha_\xi G^{(m)})(x,\xi;\lambda)\,(D^\alpha_x A)(x,\xi)\,G^{(0)}(x,\xi;\lambda).
\end{equation}
By induction, each $\partial_\xi^\alpha\partial_x^\beta G^{(\nu)}$, $\nu\ge1$, is a finite 
linear combination of terms 
\begin{equation}\label{dn01.F}
 G^{(0)}(\partial_\xi^{\alpha_1}\partial_x^{\beta_1}A)\cdot\ldots\cdot
 G^{(0)}(\partial_\xi^{\alpha_k}\partial_x^{\beta_k}A)G^{(0)}
\end{equation}
with $|\alpha_1|+\ldots+|\alpha_k|=|\alpha|+\nu$, $|\beta_1|+\ldots+|\beta_k|=|\beta|+\nu$, 
and $k\ge2$. From this we deduce the following: 

\begin{proposition}\label{dn01.7}
Let $A(x,D)$ be $\Lambda$-elliptic and 
$G^{(\nu)}(x,\xi;\lambda)=\Big(g_{ij}^{(\nu)}(x,\xi;\lambda)\Big)_{1\le i,j\le q}$ be 
defined as in \eqref{dn01.E}. In case $\nu\ge1$ we have 
 $$|\partial_\xi^\alpha\partial_x^\beta g^{(\nu)}_{ij}(x,\xi;\lambda)|\le C_{\alpha\beta}\,
   (\spk{\xi}^{r_i}+|\lambda|)^{-1}(\spk{\xi}^{r_j}+|\lambda|)^{-1}
   \spk{\xi}^{l_i+m_j-(1-\delta)\nu-|\alpha|+\delta|\beta|}$$
for all $1\le i,j\le q$, uniformly in $(x,\xi,\lambda)\in \rz^n\times\rz^n\times\Lambda$ 
$($note that the estimates 
are also valid for the elements on the diagonal, i.e., $i=j)$. 
\end{proposition}
\begin{proof}
For $n\in\nz$, let $B^{(n)}(x,\xi)=\Big(b^{(n)}_{ij}(x,\xi)\Big)_{1\le i,j\le q}$ be systems with 
$b^{(n)}_{ij}\in S^{l_i+m_j}$. The proof relies on two kinds of estimates. 

First, let $H=(B^{(3)}G^{(0)})\cdot\ldots\cdot(B^{(N)}G^{(0)})$ for an arbitrary $N\ge 3$. 
Then, by induction on $N$, it is easy to see that 
\begin{equation}\label{dn01.G}
 |h_{ij}(x,\xi;\lambda)|\le C\,(\spk{\xi}^{r_j}+|\lambda|)^{-1}\spk{\xi}^{l_i+m_j}.
\end{equation}

Second, consider $\wt H=G^{(0)}B^{(1)}G^{(0)}B^{(2)}G^{(0)}$. We shall use the explicit formula  
 $$\wt h_{ij}=\smsum_{\alpha,\beta,\gamma,\delta=1}^q 
   g^{(0)}_{i\alpha}b^{(1)}_{\alpha\beta}g^{(0)}_{\beta\gamma}b^{(2)}_{\gamma\delta}
   g^{(0)}_{\delta j}.$$
If in a summand $\beta=\gamma$, we can estimate it by 
 $$C\,\Big|g_{i\alpha}^{(0)}(x,\xi;\lambda)\spk{\xi}^{l_\alpha+m_\beta}
   (\spk{\xi}^{r_\beta}+|\lambda|)^{-1}\spk{\xi}^{l_\beta+m_\delta}
   g_{\delta j}^{(0)}(x,\xi;\lambda)\Big|$$
in view of Lemma \ref{dn03.2}. Now 
\begin{align*}
 |g_{i\alpha}^{(0)}(x,\xi;\lambda)\spk{\xi}^{l_\alpha}|
 &\le C\, 
 \begin{cases}
  (\spk{\xi}^{r_i}+|\lambda|)^{-1}\spk{\xi}^{l_i}&\quad: i=\alpha\\
  (\spk{\xi}^{r_i}+|\lambda|)^{-1}\spk{\xi}^{l_i}
  (\spk{\xi}^{r_\alpha}+|\lambda|)^{-1}\spk{\xi}^{l_\alpha+m_\alpha}
  &\quad: i=\alpha
 \end{cases}\\
 &\le C\,(\spk{\xi}^{r_i}+|\lambda|)^{-1}\spk{\xi}^{l_i}
\end{align*}
and, analogously, 
 $$|g_{\delta j}^{(0)}(x,\xi;\lambda)\spk{\xi}^{m_\delta}|\le 
   C\,(\spk{\xi}^{r_j}+|\lambda|)^{-1}\spk{\xi}^{m_j}.$$
Thus we estimate the summand by 
\begin{align*}
 C\,(\spk{\xi}^{r_i}&+|\lambda|)^{-1}(\spk{\xi}^{r_j}+|\lambda|)^{-1}
   (\spk{\xi}^{r_\beta}+|\lambda|)^{-1}\spk{\xi}^{r_\beta+l_i+m_j}\\
 &\le C\,(\spk{\xi}^{r_i}+|\lambda|)^{-1}(\spk{\xi}^{r_j}+|\lambda|)^{-1}
   \spk{\xi}^{l_i+m_j}.
\end{align*}
Arguing analogously in the case $\beta\not=\gamma$ we arrive at the estimate 
\begin{equation}\label{dn01.H}
   |\wt h_{ij}(x,\xi;\lambda)|\le C\,
   (\spk{\xi}^{r_i}+|\lambda|)^{-1}(\spk{\xi}^{r_j}+|\lambda|)^{-1}\spk{\xi}^{l_i+m_j}.
\end{equation}
Combining both estimates \eqref{dn01.G} and \eqref{dn01.H} yields 
\begin{equation}\label{dn01.I}
 |(\wt H H)_{ij}(x,\xi;\lambda)|\le C\,
 (\spk{\xi}^{r_i}+|\lambda|)^{-1}(\spk{\xi}^{r_j}+|\lambda|)^{-1}\spk{\xi}^{l_i+m_j}.
\end{equation}
To finally prove the statement of the proposition we set  
 $$B^{(n)}(x,\xi):=
   \spk{\xi}^{|\alpha_n|-\delta|\beta_n|}
   \partial_\xi^{\alpha_n}\partial_x^{\beta_n}a(x,\xi).$$
Then, according to \eqref{dn01.F}, we can represent  
$\partial^\alpha_\xi\partial^\beta_x G^{(\nu)}$ as a linear combination of terms 
 $$G^{(0)}B^{(1)}\cdot\ldots\cdot G^{(0)}B^{(k)}G^{(0)}
   \spk{\xi}^{-(1-\delta)\nu-|\alpha|+\delta|\beta|}$$
with $k\ge 2$. It remains to use the above estimate \eqref{dn01.I}. 
\end{proof}   

Using these estimates we are now in the position to construct a parametrix for 
$A(x,D)-\lambda$. For standard systems this construction can be found in \cite{Kuma}. 
However, we deal with Douglis-Nirenberg systems and also make precise the remainder estimate. 

\begin{theorem}\label{dn01.8}
There exists a $G(x,\xi;\lambda)=\Big(g_{ij}(x,\xi;\lambda)\Big)_{1\le i,j\le q}$ such that 
\begin{equation}\label{dn01.J}
 |\partial^\alpha_\xi \partial^\beta_x g_{ii}(x,\xi;\lambda)|\le 
 C_{\alpha\beta}(\spk{\xi}^{r_i}+|\lambda|)^{-1}
 \spk{\xi}^{-|\alpha|+\delta|\beta|}
\end{equation}
and, for $i\not=j$,
\begin{equation}\label{dn01.K}
 |\partial^\alpha_\xi \partial^\beta_x g_{ij}(x,\xi;\lambda)|\le 
 C_{\alpha\beta}(\spk{\xi}^{r_i}+|\lambda|)^{-1}(\spk{\xi}^{r_j}+|\lambda|)^{-1}
 \spk{\xi}^{l_i+m_j-|\alpha|+\delta|\beta|}
\end{equation} 
Moreover, for all $1\le i,j\le q$, 
\begin{align}\label{dn01.L}
\begin{split}
 |\partial^\alpha_\xi \partial^\beta_x&\{g_{ij}(x,\xi;\lambda)-g^{(0)}_{ij}(x,\xi;\lambda))^{-1}\}|\\
 & \le C_{\alpha\beta}
   (\spk{\xi}^{r_i}+|\lambda|)^{-1}(\spk{\xi}^{r_j}+|\lambda|)^{-1}
   \spk{\xi}^{l_i+m_j-(1-\delta)-|\alpha|+\delta|\beta|}
\end{split}
\end{align}
All these estimates hold uniformly in 
$(x,\xi,\lambda)\in\rz^n\times\rz^n\times\Lambda$ and for all $\alpha,\beta\in\nz^n_0$. 
Passing to the operator-level, we have 
\begin{equation}\label{dn01.M}
 \begin{split}
  G(x,D;\lambda)(A(x,D)-\lambda)&=1+R^{(0)}(x,D;\lambda),\\
  (A(x,D)-\lambda)G(x,D;\lambda)&=1+R^{(1)}(x,D;\lambda)
 \end{split}
\end{equation}
with remainders 
$R^{(k)}(x,\xi;\lambda)=\Big(r^{(k)}_{ij}(x,\xi;\lambda)\Big)_{1\le i,j\le q}$ satisfying  
\begin{equation}\label{dn01.N}
 |\partial^\alpha_\xi \partial^\beta_x r^{(k)}_{ij}(x,\xi;\lambda)| 
 \le C_{\alpha\beta N}\,\spk{\lambda}^{-1}\spk{\xi}^{-N}<\infty. 
\end{equation}
for arbitrary $N\in\nz$ and all $\alpha,\beta\in\nz^n_0$.
\end{theorem}
\begin{proof}
The symbol $G$ is defined by means of assymptotic summation as 
 $$G(x,\xi;\lambda):=G^{(0)}(x,\xi;\lambda)+\smsum_{\nu=1}^\infty 
   \chi(\eps_\nu|\xi|)G^{(\nu)}(x,\xi;\lambda),$$
where $\chi:\rz\to[0,1]$ is a smooth 0-excision 
function\footnote{i.e. $\chi$ vanishes identically in a neighborhood of 0 and $1-\chi$ is a smooth function 
with compact support}
and $\eps_1>\eps_2>\ldots\xrightarrow{j\to\infty}0$ sufficiently fast. 
By Lemma \ref{dn03.2} and Proposition \ref{dn01.7} the estimates 
\eqref{dn01.J}, \eqref{dn01.K}, and \eqref{dn01.L} then hold. 
It remains to verify \eqref{dn01.M}. To this end let us define 
\begin{align*}
 Q^N(x,\xi;\lambda)&=\sum_{\nu=0}^{N-1}G^{(\nu)}(x,\xi;\lambda),\\
 J^N(x,\xi;\lambda)&=\sum_{|\alpha|=0}^{N-1}\frac{1}{\alpha!}
   \partial^\alpha_\xi Q^N(x,\xi;\lambda)\,D^\alpha_x(A(x,\xi)-\lambda)
\end{align*} 
for $N\in\nz$. A direct computation shows that 
 $$J^N(x,\xi;\lambda)-1=\sum_{\substack{\nu<N,|\alpha|<N\\ \nu+|\alpha|\ge N}}\frac{1}{\alpha!}
   \partial^\alpha_\xi G^{(\nu)}(x,\xi;\lambda)\,D^\alpha_x A(x,\xi).$$
By Lemma \ref{dn03.2} and Proposition \ref{dn01.7} it is easily seen that then 
\begin{equation}\label{paramA}
 |\partial^\alpha_\xi\partial^\beta_x (J^N_{ij}(x,\xi;\lambda)-1)|\le C_{\alpha\beta}\,
 (\spk{\xi}^{r_i}+|\lambda|)^{-1}\spk{\xi}^{l_i+m_j-(1-\delta)N}.
\end{equation}
Let us now suppress the variables $x$ and $\xi$ from the notation. Then, for any $N$,  
\begin{align*}
  R^{(0)}(\lambda)&=G(\lambda)\#(A-\lambda)-1\\
 &=[(G(\lambda)-Q^N(\lambda))\#(A-\lambda)]  
  +[Q^N(\lambda)\#(A-\lambda)-J^N(\lambda)]+[J_N(\lambda)-1]\\
 &=: S^1(\lambda)+S^2(\lambda)+S^3(\lambda),
\end{align*}
where $\#$ denotes the Leibniz product. The construction of $G$ and Proposition \ref{dn01.7}, 
now yield that   
 $$\spk{\lambda}^{2}(G(\lambda)-Q^N(\lambda))_{ij}\,\in\,
   S^{l_i+m_j-(1-\delta)N},\qquad 1\le i,j\le q,$$
uniformly in $\lambda\in\Lambda$. From this it follows that 
 $$\spk{\lambda}S^1_{ij}(\lambda)\,\in\,S^{l_i+m_j+r_1-(1-\delta)N},\qquad 1\le i,j\le q,$$
uniformly for $\lambda\in\Lambda$. By \eqref{paramA}, the same is true for 
the components of $\spk{\lambda}S^3(\lambda)$. 
By the standard composition formula for pseudodifferential operators, we obtain 
 $$S^2(\lambda)=N\sum_{|\gamma|=N}\int_0^1\frac{(1-\theta)^{N-1}}{\gamma!}
   R^{\gamma,\theta}(\lambda)\,d\theta$$
with 
 $$R^{\gamma,\theta}(x,\xi;\lambda)=\iint e^{-iy\eta}
   \partial^\gamma_\xi Q^N(x,\xi+\theta\eta;\lambda) D^\gamma_x A(x+y,\xi)\,dy\dbar\xi,$$
where the integral has to be understood as an oscillatory integral. Employing again Proposition 
\ref{dn01.7}, it is straightforward to see that 
$\spk{\lambda}R^{\gamma,\theta}_{ij}(\lambda)\in S^{l_i+m_j-(1-\delta)|\gamma|}$ uniformly 
in $\lambda\in\Lambda$ and $0\le\theta\le1$. This clearly implies that 
$\spk{\lambda}S^2_{ij}(\lambda)\in S^{l_i+m_j-(1-\delta)N}$ uniformly in $\lambda\in\Lambda$. 
Since $N$ was arbitrary, it follows that $\spk{\lambda}R^{(0)}_{ij}(\lambda)\in S^{-\infty}$ 
uniformly in $\lambda\in\Lambda$. For $R^{(1)}$ one can argue analogously by constructing a 
right-parametrix to $A(x,D)-\lambda$ and the using that this coincides with 
$G(x,D;\lambda)$ up to a smoothing remainder (which also has the requested decay in $\lambda$). 
\end{proof}

\subsection{Diagonalization}\label{sec:dn03.1}

The following theorem states, roughly speaking, that each elliptic system can be 
transformed to diagonal form via conjugation with a suitable isomorphism. This transformation 
also preserves $\Lambda$-ellipticity. The theorem was first proved by Kozhevnikov \cite{Kozh2}
for systems on compact manifolds. We follow his proof but extend his result both to operators 
on $\rz^n$ and more general symbol classes. 

\begin{theorem}\label{dn03.1}
Let $A(x,D)$ be $\Lambda$-elliptic in the sense of Definition {\rm\ref{dn02.5}}. 
Then there exists a $(q\times q)$-matrix 
 $$V(x,D)=\Big(v_{ij}(x,D)\Big)_{1\le i,j\le q}$$
with $v_{ii}\equiv 1$ and 
 $$v_{ij}(x,\xi)\in
   S^{-m_i+m_j}_\delta(\rz\times\rz)\,\cap\,S^{l_i-l_j}_\delta(\rz\times\rz)\;\footnotemark$$
such\footnotetext{Note that $-m_i+m_j<l_i-l_j$ if $i<j$, and $l_i-l_j<-m_i+m_j$ if $i>j$.} 
that $V(x,D)$ is invertible and 
 $$V(x,D)^{-1}A(x,D)V(x,D)
   =\mathrm{diag}\Big(\wt a_{11}(x,D),\ldots,\wt a_{qq}(x,D)\Big)+R^{(-\infty)}(x,D),$$
where 
\begin{itemize}
 \item[i$)$] $r^{(-\infty)}_{ij}\in S^{-\infty}(\rz\times\rz)$ for all $1\le i,j\le q$, 
 \item[ii$)$] each $\wt a_{ii}(x,\xi)\in S^{r_i}_\delta(\rz\times\rz)$ is $\Lambda$-elliptic, i.e. 
   $$|(\wt a_{ii}(x,\xi)-\lambda)\big|\ge C\,(\spk{\xi}^{r_i}+|\lambda|)
   \qquad \forall\;x\in\rz\quad\forall\;|\xi|\ge R\quad\forall\;\lambda\in\Lambda.$$
\end{itemize}
\end{theorem}

Before we come to the proof, let us clarify that the invertibility of $V(x,D)$ refers to all 
induced operators 
$\mathop{\mbox{$\oplus$}}\limits_{j=1}^q H^{s-l_j}_p\to
\mathop{\mbox{$\oplus$}}\limits_{i=1}^q H^{s-l_i}_p$ as well as 
$\mathop{\mbox{$\oplus$}}\limits_{j=1}^q H^{s+m_j}_p\to
\mathop{\mbox{$\oplus$}}\limits_{i=1}^q H^{s+m_i}_p$ for arbitrary $s\in\rz$ and $1<p<\infty$. 
By spectral invariance of pseudodifferential operators (see \cite{LeSc}, for example), the inverse is again 
of the form $W(x,D)=\big(w_{ij}(x,D)\big)_{1\le i,j\le q}$ with  
 $$w_{ij}(x,\xi)\in
   S^{-m_i+m_j}_\delta(\rz\times\rz)\,\cap\,S^{l_i-l_j}_\delta(\rz\times\rz)$$
and, for a suitable 0-excision function $\chi(\xi)$, 
\begin{equation}\label{dn03.A.5}
  \chi(\xi)\big(V(x,\xi)^{-1}-W(x,\xi)\big)_{ij}\in
  S^{-m_i+m_j-(1-\delta)}_\delta(\rz\times\rz)\,\cap\,S^{l_i-l_j-(1-\delta)}_\delta(\rz\times\rz).
\end{equation}

\begin{proof}[Proof of Theorem \ref{dn03.1}]
It shall be more convenient to consider instead of $A(x,D)$ the system 
 $$B:=L(x,D)A(x,D)L(x,D)^{-1}=\Big(b_{ij}(x,D)\Big)_{1\le i,j\le q},$$
where $L(x,D)=\mathrm{diag}(\spk{D}^{-l_1},\ldots,\spk{D}^{-l_q})$. Then we have 
 $$b_{ij}(x,D)=\spk{D}^{-l_i}a_{ij}(x,D)\spk{D}^{l_j}\in S^{r_j}_\delta(\rz\times\rz).$$  
{\bf Step 1:} \ In the first part of the proof, we construct operators 
 $$S=\Big(s_{ij}(x,D)\Big)_{1\le i,j\le q}, \qquad 
   D=\mathrm{diag}\Big(d_{11}(x,D),\ldots,d_{qq}(x,D)\Big)$$
such that 
\begin{equation}\label{dn03.A}
 BS\equiv SD\quad\mathrm{mod}\quad S^{-\infty}
\end{equation}
and $s_{ii}(x,\xi)\equiv 1$. These operators will be obtained by the Ansatz 
 $$S=\smsum_{\nu=0}^\infty S^{(\nu)},\qquad D=\smsum_{\mu=0}^\infty D^{(\mu)}$$
(in the sense of asymptotic summation of pseudodifferential operators) with diagonal 
matrices $D^{(\mu)}$, and 
 $$s_{ij}^{(\nu)}(x,\xi)\in S^{\min(0,r_j-r_i)-\nu}_\delta,\qquad 
   d_{ii}^{(\mu)}(x,\xi)\in S^{r_i-\mu}_\delta.$$
In the following we shortly write 
 $$B_{ij}:=b_{ij}(x,D),\qquad S^{(\nu)}_{ij}:=s_{ij}^{(\nu)}(x,D),\qquad 
   D_{ii}^{(\nu)}:=d_{ii}^{(\nu)}(x,D).$$
Now fix an arbitrary $j\in\{1,\ldots,q\}$. Then 
\begin{align}
 B_{im}S_{mj}^{(\nu)}\in S^{r_j+\min(0,r_m-r_j)-\nu}_\delta,\qquad 
 S_{ij}^{(\nu)}D_{jj}^{(\mu)}\in S^{r_j+\min(0,r_j-r_i)-\nu-\mu}_\delta.
\end{align}
Thus, using the above Ansatz, the statement   
 $$(BS)_{ij}\equiv (SD)_{ij}\quad\mathrm{mod}\quad S^{r_j-N-\eps}_\delta,\qquad 
   \eps:=\min_{1\le k\le q-1}(r_k-r_{k+1}),$$
with $N\in\nz_0$ is equivalent to 
\begin{equation}\label{dn03.B}
 \smsum_{\nu=0}^{N}\smsum_{m=1}^j B_{im}S^{(\nu)}_{mj}
 -\smsum_{\mu+\nu\le N}S_{ij}^{(\nu)}D_{jj}^{(\mu)}\;\in\;S^{r_j-N-\eps}_\delta.
\end{equation}
We now show that we can iteratively construct operators $S^{(0)},S^{(1)},\ldots$ and 
$D^{(0)},D^{(1)},\ldots$ with components of the required order and such that the expression 
in \eqref{dn03.B} equals zero modulo $S^{-\infty}$. 
 
In fact, for $N=0$, to obtain zero in \eqref{dn03.B} is equivalent to
\begin{align}
 \smsum_{m=1}^j B_{im}S^{(0)}_{mj}&=0, & i<j,\label{dn03.C}\\
 \smsum_{m=1}^j B_{im}S^{(0)}_{mj}-S_{ij}^{(0)}D^{(0)}_{jj}&=0, & i\ge j,\label{dn03.D}
\end{align}
If we set $S_{jj}^{(0)}=1$, then \eqref{dn03.C} and the equation for $i=j$ from 
\eqref{dn03.D} can be written in the following form: 
\begin{equation}\label{dn03.E}
   B[j]
   \begin{pmatrix}
    S_{1j}^{(0)}\\ \vdots\\ S_{k-1,j}^{(0)}\\1
   \end{pmatrix}
   =
   \begin{pmatrix}
    0\\ \vdots\\ 0\\ D_{jj}^{(0)}
   \end{pmatrix}
   \iff
   \begin{pmatrix}
     B_{11}&\ldots & B_{1,j-1}& 0\\
     \vdots&\ddots & \vdots&\vdots\\
     B_{j-1,1}&\ldots & B_{j-1,j-1}& 0\\
    B_{j1}&\ldots&B_{j,j-1}&-1 
   \end{pmatrix}
   \begin{pmatrix}
    S_{1j}^{(0)}\\ \vdots\\ S_{j-1,j}^{(0)}\\D_{jj}^{(0)}
   \end{pmatrix}
   =-
   \begin{pmatrix}
    B_{1j}\\ \vdots\\ \vdots\\ B_{jj}
   \end{pmatrix}.
\end{equation}
However, since $B[j-1]$ is elliptic by assumption, this system determines  
$S_{1j}^{(0)},\ldots,S_{j-1,j}^{(0)}, D_{jj}^{(0)}$ uniquely (up to $S^{-\infty}$). 
Moreover, by Cramer's rule we obtain 
 $$\chi(\xi)d_{jj}^{(0)}(x,\xi)
   =\chi(\xi)\frac{\mathrm{det}B[j](x,\xi)}{\mathrm{det}B[j-1](x,\xi)}
   \quad\mathrm{mod}\quad S^{r_j-(1-\delta)}_\delta,$$
with a suitable $0$-excision function $\chi$. Therefore $d_{jj}^{(0)}(x,\xi)$ is elliptic of 
order $r_j$ and we can determine $S_{j+1,j}^{(0)},\ldots,S_{qj}^{(0)}$ from the remaining 
equations of \eqref{dn03.D}. 

Now assume $S^{(0)},\ldots,S^{(N-1)}$ and $D^{(0)},\ldots,D^{(N-1)}$ have been determined for 
some $N\in\nz$. If we then denote by $R_{ij}^{(N-1)}$ the sum of all summands in \eqref{dn03.B} 
which are determined, the expression in\eqref{dn03.B} equals zero if and only if 
\begin{align}
 \smsum_{m=1}^j B_{im}S^{(N)}_{mj}&=-R_{ij}^{(N-1)}, & i<j,\qquad\qquad\label{dn03.F}\\
 \smsum_{m=1}^j B_{im}S^{(N)}_{mj}-S_{ij}^{(0)}D^{(N)}_{jj}-S_{ij}^{(N)}D^{(0)}_{jj}
   &=-R_{ij}^{(N-1)}, & i\ge j. 
   \qquad\qquad\label{dn03.G}
\end{align}
Setting $S_{jj}^{(N)}=0$, \eqref{dn03.F} together with the equation for $i=j$ from 
\eqref{dn03.G} is equivalent to 
 $$\begin{pmatrix}
     B_{11}&\ldots & B_{1,j-1}& 0\\
     \vdots&\ddots & \vdots&\vdots\\
     B_{j-1,1}&\ldots & B_{j-1,j-1}& 0\\
    B_{j1}&\ldots&B_{j,j-1}&-1 
   \end{pmatrix}
   \begin{pmatrix}
    S_{1j}^{(N)}\\ \vdots\\ S_{j-1,j}^{(N)}\\D_{jj}^{(N)}
   \end{pmatrix}
   =-
   \begin{pmatrix}
    R_{1j}^{(N-1)}\\ \vdots\\ \vdots\\ R_{jj}^{(N-1)}
   \end{pmatrix}
   .$$
By this system $D_{jj}^{(N)}$ and the $S_{ij}^{(N)}$ for $i\le j$ are uniquely determined, 
up to smoothing operators. The remaining $S_{ij}^{(N)}$, $i>j$, are then determined 
by \eqref{dn03.G}. 

{\bf Step 2:} \ The next step is to verify the $\Lambda$-ellipticity of $D_{jj}^{(0)}$. 
To this end we insert 
the parameter $\lambda\in\Lambda$ in the equation for $i=j$ of \eqref{dn03.D}, writing 
 $$\smsum_{m=1}^{j-1} B_{im}S^{(0)}_{mj}+(B_{jj}-\lambda)-(D^{(0)}_{jj}-\lambda)=0.$$
Arguing similarly as above in \eqref{dn03.E}, we obtain 
 $$\chi(\xi)\big(d_{jj}^{(0)}(x,\xi)-\lambda\big)=
   \chi(\xi)\frac{\mathrm{det}\big(B[j](x,\xi)-\lambda E_j\big)}{\mathrm{det}B[j-1](x,\xi)}
   \quad\mathrm{mod}\quad S^{r_j-(1-\delta)}_\delta,$$
with a remainder independent of $\lambda$.\footnote{More precisely, one obtains a system 
analogous to \eqref{dn03.E}, replacing $D_{jj}^{(0)}$ and $B_{jj}^{(0)}$ by 
$D_{jj}^{(0)}-\lambda$ and $B_{jj}^{(0)}-\lambda$, respectively. One then has to observe that 
the symbol of the operator on the left-hand side differs by a lower order term, which does not 
depend on $\lambda$, from the pointwise product of the respective symbols.} 

{\bf Step 3:} \ We shall modify $S$ by smoothing terms in such a way that $S$ is invertible. 
To this end we decompose $S$ in its lower left, upper right, and diagonal part, i.e. 
$S=1+L+U$ with   
 $$L=\begin{pmatrix}
        0&\ldots&\ldots&0\\
        S_{21}&\ddots&&\vdots\\
        \vdots&\ddots&\ddots&\vdots\\
        S_{q1}&\ldots&S_{q,q-1}&0
       \end{pmatrix},\qquad
   U=\begin{pmatrix}
        0&S_{12}&\ldots&S_{1q}\\
        \vdots&\ddots&\ddots&\vdots\\
        \vdots& &\ddots& S_{q-1,q}\\
        0&\ldots&\ldots&0
       \end{pmatrix}.$$
Moreover, let $\sigma(\xi)$ be a 0-excision function, and 
 $$\Sigma(\eps)=\sigma(\eps D),\qquad 0<\eps \le 1.$$
We then obtain that 
 $$L^\prime(\eps):=1+L\Sigma(\eps):X\lra X$$
is an isomorphism both for $X=\mathop{\oplus}\limits_{j=1}^q H^{s}_p(\rz^n)$ and 
$X=\mathop{\oplus}\limits_{j=1}^q H^{s+r_j}_p(\rz^n)$, with inverse 
 $$L^\prime(\eps)^{-1}=1-L\Sigma(\eps)+\ldots+(-L\Sigma(\eps))^{q-1}.$$
Since this operator is a lower left triangular matrix with $1$'s on the diagonal and 
$\{\sigma(\eps\xi)\st 0<\eps\le 1\}$ is a bounded subset of $S^0(\rz^n_\xi)$, 
there exists a constant $C\ge 1$ such that 
\begin{equation}\label{dn03.H}
 1\le\|L^\prime(\eps)^{-1}\|_{\calL(X)}\le C\qquad\forall\;0<\eps\le1.
\end{equation}
Since each component $u_{ij}(x,D)$ of $U$ has strictly negative order by construction of $S$, 
it follows that 
 $$u_{ij}(x,\xi)\sigma(\rho\xi)\xrightarrow{\rho\to 0}0\quad\text{in}\quad S^0_\delta.$$
This together with \eqref{dn03.H} allows us to choose $0<\rho^*\le 1$ such that 
 $$L^\prime(\eps)+U\Sigma(\rho^*):
   \mathop{\oplus}_{j=1}^q H^{s}_p(\rz^n)\lra\mathop{\oplus}_{j=1}^q H^{s}_p(\rz^n)$$
is an isomorphism for any $0<\eps\le 1$. Arguing in an analogous 
way\footnote{Note that the nontrivial components $l_{ij}(x,D)$ of $L$ are of order $0$ and $0<r_j-r_i$.} 
for the operator 
 $$U^\prime(\rho):=1+U\Sigma(\rho):X\lra X,$$
we find $0<\eps^*\le 1$ such that 
 $$U^\prime(\rho)+L\Sigma(\eps^*):
   \mathop{\oplus}_{j=1}^q H^{s+r_j}_p(\rz^n)\lra\mathop{\oplus}_{j=1}^q H^{s+r_j}_p(\rz^n)$$
is an isomorphism for any $0<\rho\le 1$. It follows that 
 $$\wt S:=1+L\Sigma(\eps^*)+U\Sigma(\rho^*):X\lra X$$
is an isomorphism (for both choices of $X$). Moreover, 
 $$S-\wt S=L(1-\sigma)(\eps^*D)+U(1-\sigma)(\rho^*D)$$
is a smoothing operator, since $(1-\sigma)(\xi)$ is compactly supported. 

{\bf Step 4:} \ In view of \eqref{dn03.A} and Step 3, we may assume that 
 $$BS=SD+R_1=S(D+S^{-1}R_1)$$
for some smoothing operator $R_1$. Due to the spectral invariance of pseudodifferential 
operators, also $S^{-1}R_1$ is a smoothing operator. Thus $BS=S\wt{D}$ with a $\wt{D}$ that 
differs from $D$ of Step 1 by a smoothing operator. Defining now 
 $$V(x,D):=L(x,D)^{-1}\,S\,L(x,D), \qquad \wt{A}(x,D):=L(x,D)^{-1}\,D\,L(x,D),$$
and using that $B=L(x,D)\,A(x,D)\,L(x,D)^{-1}$ by definition, we obtain that 
 $$V(x,D)^{-1}\,A(x,D)\,V(x,D)=\wt{A}(x,D)\quad\mathrm{mod}\,S^{-\infty},$$
and $V$ as well as the diagonal matrix $\wt{A}$ have the properties described in the theorem. 
Thus the proof is complete.
\end{proof}

\begin{corollary}\label{dn03.1.5}
$\Lambda$-ellipticity in the sense of Definition {\rm\ref{dn02.5}} implies 
$\Lambda$-ellipticity in the sense of Definition {\rm\ref{dn02.4}}.
\end{corollary}
\begin{proof}
Let $A(x,D)$ be $\Lambda$-elliptic in the sense of Definition {\rm\ref{dn02.5}}. 
We use the notation of Theorem \ref{dn03.1}. Let us set 
 $$W(x,D):=V(x,D)^{-1},\qquad \wt{A}(x,D)=\mathrm{diag}\Big(\wt a_{11}(x,D),\ldots,\wt a_{qq}(x,D)\Big).$$ 
Using the standard property that 
$b_1\#b_2-b_1b_2\in S^{\mu_1+\mu_2-(1-\delta)}_\delta$ for symbols $b_j(x,\xi)\in S^{\mu_j}_\delta$ together 
with \eqref{dn03.A.5}, it is easily verified that  
 $$V(x,\xi)^{-1}A(x,\xi)V(x,\xi)=\wt{A}(x,\xi)+R(x,\xi)$$
for sufficiently large $|\xi|$, with a remainder $R(x,\xi)=\Big(r_{ij}(x,\xi)\Big)_{1\le i,j\le q}$ 
satisfying $r_{ij}(x,\xi)\in S^{l_i+m_j-(1-\delta)}_\delta$. It follows that  
 $$\mathrm{det}(A(x,\xi)-\lambda)=\mathrm{det}\big(\wt{A}(x,\xi)+R(x,\xi)-\lambda\big).$$
By Theorem \ref{dn03.1} it is obvious that $\wt{A}$ is $\Lambda$-elliptic in the sense of Definition \ref{dn02.4}. 
By Corollary \ref{dn03.2.5} this is then also true for $\wt{A}+R$, hence for $A$. 
\end{proof}

\forget{
\begin{remark}
An alternative approach to construct the parametrix for a $\Lambda$-elliptic system is to 
first transform this system to diagonal form $($using Theorem {\rm\ref{dn03.1}}$)$, construct then the 
parametrix, and finally to transform back. However, in this passage one seems to loose some 
information, e.g. the difference in the estimates of diagonal and nondiagonal entries. 
Thus Theorem {\rm\ref{dn01.8}} seems to be the more accurate construction.
\end{remark}
}

\section{Bounded $H_\infty$-calculus for perturbed Douglis-Nirenberg systems}\label{sec:dn04}

Throughout this section, we let $A(x,D)$ be a $\Lambda$-elliptic Douglis-Nirenberg system, 
and we consider the unbounded operator 
\begin{equation}\label{dn04.C}
 \calA=A(x,D)+K:\,\calD\subset\calH\lra\calH
\end{equation}
where 
\begin{equation}\label{dn04.A}
 \calD=\mathop{\mbox{\Large$\oplus$}}_{j=1}^q H^{s+m_j}_p(\rz^n),\qquad
 \calH=\mathop{\mbox{\Large$\oplus$}}_{i=1}^q H^{s-l_i}_p(\rz^n)
\end{equation}
with arbitrary fixed $s\in\rz$ and $1<p<\infty$, and $K=(K_{ij})_{1\le i,j\le q}$ is 
a perturbation satisfying, for some $\eps>0$,  
\begin{equation}\label{dn04.B}
 K_{ij}:H^{s+m_j-\eps}(\rz^n)\lra H^{s-l_i}(\rz^n)\qquad\forall\;1\le i,j\le q. 
\end{equation}
We shall show that then $\calA$ generates an analytic semigroup and, even stronger, that it admits a 
bounded $H_\infty$-calculus. 

\subsection{Resolvent estimate}\label{sec:dn04.1}

From standard elliptic theory (i.e. the existence of a parametrix to $A(x,D)$ which inverts 
$A(x,D)$ up to smoothing remainders), it is straightforward 
to deduce that the operator $\calA$ from \eqref{dn04.C} is closed. 

\begin{theorem}\label{dn04.1}
Let $\calA$ be as in \eqref{dn04.C}. Then there exists an $\alpha_0\ge 0$ such that for each $\alpha\ge\alpha_0$ the 
resolvent of $\calA_\alpha:=\calA+\alpha$ exists on $\Lambda$ and satisfies 
\begin{equation}\label{dn02.D}
 \|(\calA_\alpha -\lambda)^{-1}\|_{\calL(\calH)}\le 
 C\,\spk{\lambda}^{-1}\qquad\forall\;\lambda\in\Lambda.
\end{equation}
Moreover, with the notation from Theorem {\rm\ref{dn01.8}},
\begin{equation}\label{dn02.E}
 (\calA_\alpha -\lambda)^{-1}=G(x,D;\lambda)+R(\lambda)\qquad\forall\;\lambda\in\Lambda
\end{equation}
with a remainder $R(\lambda)=\Big(R_{ij}(\lambda)\Big)_{1\le i,j\le q}$ satisfying, 
for some $\eps>0$,  
\begin{equation}\label{dn02.F}
 \|R(\lambda)\|_{\calL(\calH)}\le 
 C\,\spk{\lambda}^{-1-\eps}\qquad\forall\;\lambda\in\Lambda. 
\end{equation}
\end{theorem}

Before we come to the proof, let us remark that in \eqref{dn02.E} the operator 
$G(x,D;\lambda)$ is constructed as in Section \ref{sec:dn03.2}, but with respect to the 
symbol $A(x,\xi)+\alpha$ (recall Remark \ref{dn02.7}). 
Moreover, the estimates in Theorem \ref{dn01.8} imply that 
 $$|\partial^\alpha_\xi\partial^\beta_x g_{ij}(x,\xi;\lambda)|\le C_{\alpha\beta\tau}
 \spk{\lambda}^{-1+\tau}\spk{\xi}^{-m_i-l_j-(1-\tau)r_i-|\alpha|+\delta|\beta|}
 \qquad\forall\;0\le\tau\le 1,$$
uniformly in $(x,\xi,\lambda)\in\rz^n\times\rz^n\times\Lambda$. Thus, by Theorem \ref{dn02.2}, 
\begin{equation}\label{dn02.G}
 \|g_{ij}(x,D;\lambda)\|_{\calL\left(H^{\sigma-l_j}(\rz^n),H^{\sigma+m_i-(1-\tau)r_i}(\rz^n)\right)}
 \le C_{\sigma,\tau}\,\spk{\lambda}^{-1+\tau}
 \qquad\forall\;0\le\tau\le 1\quad\forall\;\sigma\in\rz.
\end{equation}

\begin{proof}[Proof of Theorem \ref{dn04.1}]
Choose $\alpha_0$ so large that $A(x,D)+\alpha$ is $\Lambda$-elliptic and let 
$\wt\calA=\calA+\alpha-K$. 
Then Theorem \ref{dn01.8} implies that there exists a $c\ge 0$ such that 
 $$(\wt\calA-\lambda)^{-1}=\big(1+R^{(0)}(x,D;\lambda)\big)^{-1}G(x,D;\lambda)
   =G(x,D;\lambda)\big(1+R^{(1)}(x,D;\lambda)\big)^{-1}$$
for all $\lambda\in\Lambda$ with $|\lambda|\ge c$. 
Since $\|\big(1+R^{(k)}(x,D;\lambda)\big)^{-1}\|$ is uniformly 
bounded in $|\lambda|\ge c$, we derive from \eqref{dn02.G} with $\tau=0$ and $\sigma=s$ that  
 $$\|(\wt\calA-\lambda)^{-1}\|\le C\,\spk{\lambda}^{-1}\qquad\forall\;|\lambda|\ge c.$$
By definition of $\wt\calA$ we can write 
 $$\calA_\alpha-\lambda=\big(1+K(\wt\calA-\lambda)^{-1}\big)(\wt\calA-\lambda)
   =:\big(1+S(\lambda)\big)(\wt\calA-\lambda).$$
Using the above representation of the resolvent, 
\begin{align*}
 S(\lambda)&=KG(x,D;\lambda)\big(1+R^{(1)}(x,D;\lambda)\big)^{-1}.
\end{align*}
Now choose $0\le\tau<1$ such that $\tau\ge\frac{r_l-\eps}{r_l}$ whenever $r_l>0$, i.e. $\tau$ fulfills 
$(1-\tau)r_l\le\eps$ for all $1\le l\le q$. Then \eqref{dn02.G} with $\sigma=s$ together with assumption \eqref{dn04.B} yield that  
 $$\spk{\lambda}^{1-\tau}K_{il}g_{lj}(x,D;\lambda): 
   H^{s-l_j}_p(\rz^n)\lra H^{s-l_i}_p(\rz^n),\qquad 1\le i,j,l\le q,$$
is unifomly bounded in $\lambda\in\Lambda$. It follows that 
$\spk{\lambda}^{1-\tau}S(\lambda)\in\calL(\calH)$ is uniformly bounded in $\lambda\in\Lambda$. 
We conclude that the resolvent 
$(\calA_\alpha-\lambda)^{-1}$ exists for all $|\lambda|\ge c$ for a sufficiently large 
constant $c$. Replacing now $\alpha_0$ from the beginning of the proof by $\alpha_0+c$, 
the resolvent exists for all $\lambda\in\Lambda$. 

Representation \eqref{dn02.E} we derive by repeated use of the formula 
$(1+T)^{-1}=1-T+T(1+T)^{-1}T$. In case $K=0$ we apply this formula to $T=R^{(0)}$, and even 
obtain a remainder of decay $O(\spk{\lambda}^{-2})$ with respect to the operator-norm in 
$\calH$. Otherwise, we have 
 $$(\calA_\alpha-\lambda)^{-1}=G(x,D;\lambda)\big(1+R^{(1)}(x,D;\lambda)\big)^{-1}
   \big(1+S(\lambda)\big)^{-1}
   \equiv G\big(1+S(\lambda)\big)^{-1}+O(\spk{\lambda}^{-2}).$$
Furthermore, 
\begin{align*}
 G(x,D;\lambda)\big(1+S(\lambda)\big)^{-1}&=G(x,D;\lambda)+G(x,D;\lambda)
   \Big(-S(\lambda)+S(\lambda)\big(1+S(\lambda)\big)^{-1}S(\lambda)\Big)\\
 &\equiv G(x,D;\lambda)+O(\spk{\lambda}^{-2+\tau}).
\end{align*}
This finishes the proof. 
\end{proof}

\begin{corollary}\label{dn04.2}
Let $\Lambda=\Lambda(\theta)$ with $\theta<\frac{\pi}{2}$. Choosing $\alpha\ge0$ large enough, 
$\calA+\alpha$ is the infinitesimal generator of a bounded analytic semigroup on $\calH$. 
\end{corollary}

\subsection{Short review of the $H_\infty$-calculus}\label{sec:dn04.2}

Let us recall some basic facts about the $H_\infty$-calculus for a closed, densely defined 
operator 
 $$A:\calD(A)\subset X\longrightarrow X$$
in a Banach space $X$. This calculus was originally introduced by McIntosh \cite{Mcin}. 
We refer to \cite{KuWe} for a detailed presentation. 
Given $0<\theta<\pi$, let $\Lambda$ be as in \eqref{dn01.C} and 
$\partial\Lambda=\partial\Lambda(\theta)$ its parameterized boundary. 
Assume that 
\begin{itemize}
 \item $\Lambda\setminus\{0\}$ is contained in the resolvent set of $A$,
 \item $\|\lambda(\lambda-A)^{-1}\|_{\calL(X)}$ is uniformly bounded in 
  $0\not=\lambda\in\Lambda$,
 \item $A$ is injective with dense range. 
\end{itemize}

We let $H_\infty=H_\infty(\theta)$ denote the space of all functions 
$f:\cz\setminus\Lambda\to\cz$ which are 
holomorphic and bounded, equipped with the supremum norm. The subspace $H=H(\theta)$ 
consists of all functions which additionally satisfy, for some $s>0$, 
 $$\sup_{\lambda\in \cz\setminus\Lambda}(|\lambda|^{-s}+|\lambda|^s)
   |f(\lambda)|<\infty.$$
This subspace is dense in $H_\infty$ in the following sense: Given $f\in H_\infty$, 
there exists a sequence $(f_j)_{j\in\nz}\subset H$ such 
that $f_j\to f$ locally uniformly on compact subsets of $\cz\setminus\Lambda$, and 
$\|f_j\|_\infty\le c\,\|f\|_\infty$ for some constant $c$ which is independent of 
$j\in\nz$. 
Moreover, each $f\in H_\infty$ possesses $($non-tangential$)$ boundary values that define  
$f|_{\partial\Lambda}\in L_\infty(\partial\Lambda)$.

Because of the decay property, for every $f\in H$ the integral 
 \begin{equation}\label{ha}
   f(A):=\frac{1}{2\pi i}\int_{\partial\Lambda} f(\lambda)(\lambda-A)^{-1}\,d\lambda
 \end{equation}
converges absolutely in the $\calL(X)$-norm and thus defines an operator $f(A)\in\calL(X)$. 
By approximation, the definition of $f(A)$ can be extended to all $f\in H_\infty$: 
If $(f_j)_{j\in\nz}\subset H$ is an approximating sequence as described above, the limit 
 $$f(A)x=\lim_{j\to\infty}f_j(A)x$$
exists for all $x\in\calD(A)$ and does not depend on the specific choice of the sequence. 
The resulting operator $f(A):\calD(A)\subset X\to X$ is closable. 
Its closure will be denoted again by $f(A)$. 

\begin{definition}\label{dn04.3}
The operator $A$ admits a {\em bounded $H_\infty$-calculus} for the sector 
$\cz\setminus\Lambda$ if $f(A)\in\calL(X)$ for any $f\in H$ and, 
with some constant $M\ge0$, 
\begin{equation}\label{hA}
 \|f(A)\|_{\calL(X)}\le M\,\|f\|_\infty \qquad\forall\;f\in H.
\end{equation}
\end{definition}

If $A$ admits a bounded $H_\infty$-calculus with respect to $\cz\setminus\Lambda$ then, 
due to Banach-Steinhaus theorem, the estimate \eqref{hA} extends to all $f\in H_\infty$. 

We finish this summary with a simple observation of which we shall make use in the next section. 

\begin{remark}\label{dn04.4}
The $H_\infty$-calculus is invariant under conjugation with isomorphisms, i.e. if 
$V\in\calL(X)$ is an isomorphism, then $A$ admits a bounded $H_\infty$-calculus with respect 
to $\cz\setminus\Lambda$ if and only if $B:=V^{-1}A{V}$ with $\calD(B)=V^{-1}(\calD(A))$ does. 
In this case, 
 $$f(B)=V^{-1}f(A)\,V\qquad\forall\;f\in H_\infty.$$ 
\end{remark}

\forget{
\begin{proof}
Since $\wt{V}^{-1}=V^{-1}|_{\calD(A)}$, it is obvious that 
$(\lambda-B)=V^{-1}(\lambda-A)V:\calD(A)\to X$ is bijective with inverse 
 $$(\lambda-B)^{-1}=\wt{V}^{-1}(\lambda-A)^{-1}V=V^{-1}(\lambda-A)V.$$
Since $\|\cdot\|_X$ and $\|V\cdot\|_X$ are equivalent norms on $X$, this immediately 
shows the resolvent estimate. The formula for $f(B)$ follows by inserting the resolvent 
expression into the above Dunford integral, followed by pulling out $V$ and $V^{-1}$ from 
the integral (which is justified by the absolute convergence of the integral). 
\end{proof}
}

\subsection{Douglis-Nirenberg systems}\label{sec:dn04.3}

We shall improve Corollary \ref{dn04.2}: 

\begin{theorem}\label{dn04.5}
Let $\calA$ be as in \eqref{dn04.C}. Then there exists an 
$\alpha_0\ge 0$ such that for each $\alpha\ge\alpha_0$ the operator 
$\calA_\alpha=\calA+\alpha$ admits a bounded $H_\infty$-calculus with respect to $\Lambda$ in $\calH$. 
\end{theorem}
\begin{proof}
According to Theorem \ref{dn03.1} and Lemma \ref{dn04.3} we may assume that $A(x,D)$ is 
a system of diagonal form. Replacing from the very beginning $A(x,D)$ by $A(x,D)+\alpha$ for 
$\alpha\ge\alpha_0$ as in Theorem \ref{dn04.1}, we may assume $\alpha=0$ and obtain that 
 $$(\calA-\lambda)^{-1}=\mathrm{diag}
   \Big({g}_{11}(x,D;\lambda),\ldots,{g}_{qq}(x,D;\lambda)\Big)+R(\lambda),$$
where $\spk{\lambda}^{1+\eps}R(\lambda)\in\calL(\calH)$ is uniformly bounded in 
$\lambda\in\Lambda$ for some $\eps>0$, and $g_{ii}(x,D;\lambda)$ is the parametrix to 
${a}_{ii}(x,D)-\lambda$ in the sense of Theorem \ref{dn01.8} (in the special case of $q=1$). 
We now insert this representation in the Dunford integral \eqref{ha}, obtaining two summands, namely 
 $$R(f):=\frac{1}{2\pi i}\int_{\partial\Lambda}f(\lambda)R(\lambda)\,d\lambda$$
and a diagonal matrix $G(f)$ with entries 
 $$G_{ii}(f)=\frac{1}{2\pi i}\int_{\partial\Lambda}f(\lambda)g_{ii}(x,D;\lambda)\,d\lambda,
   \qquad i=1,\ldots,q.$$
Since $R(\lambda)$ is an integrable function with values in $\calL(\calH)$, it is obvious 
that we can estimate 
 $$\|R(f)\|_{\calL(\calH)}\le M\,\|f\|_\infty\qquad\forall\;f\in H$$
with a constant $M$ independent of $f$. To show the analogous estimate for $G(f)$ we have to 
verify that
 $$\|G_{ii}(f)\|_{\calL(H^{s-l_i}(\rz^n))}\le M\,\|f\|_\infty\qquad\forall\;f\in H$$
for any $1\le i\le q$. This has been done already in the proof of Theorem 3.11 of \cite{EsSe}. 
For convenience of the reader, we shortly sketch the argument: Write 
 $$\wt{g}_{ii}(x,\xi;\lambda):=g_{ii}(x,\xi;\lambda)-g_{ii}^{(0)}(x,\xi;\lambda)$$
(cf.\ Theorem \ref{dn01.8}). Correspondingly, $G_{ii}(f)=G_{ii}^{(0)}(f)+\wt{G}_{ii}(f)$ 
with obvious meaning of notation. Then 
$G_{ii}^{(0)}(f)=a_f^{(0)}(x,D)$ and $\wt{G}_{ii}(f)=\wt{a}_f^{(0)}(x,D)$ with 
\begin{align*}
 a_f^{(0)}(x,\xi)
   &=\frac{1}{2\pi i}\int_{\scrC(\xi)}f(\lambda)(a_{ii}(x,\xi)-\lambda)^{-1}\,d\lambda,\\
 \wt{a}_f^{(0)}(x,\xi)
   &=\frac{1}{2\pi i}\int_{\partial\Lambda}f(\lambda)\wt{g}_{ii}(x,\xi;\lambda)\,d\lambda.
\end{align*}
Here, $\scrC(\xi)$ is a path of ``pac-man shape" consisting of the circular part 
$\{\lambda\in\cz\setminus\Lambda\st |\lambda|=c\spk{\xi}^{r_i}\}$ and the two straight lines 
$\{\lambda\in\partial\Lambda\st |\lambda|\le c\spk{\xi}^{r_i}\}$, where  
$c=2\|a_{ii}\|^{r_i}_{\delta,0}$ (cf. the notation given after Definition \ref{dn02.1}).
Then it is straightforward to see that 
$\big\{a_f^{(0)}\big/\|f\|_\infty\st 0\not=f\in H\big\}$ is a bounded 
subset of $S^0_\delta$ and, using the estimate \eqref{dn01.L} for $i=j$, that 
$\big\{\wt{a}_f\big/\|f\|_\infty\st 0\not=f\in H\big\}$ is a bounded subset of $S^{1-\delta}_\delta$. 
In particular, the sets of associated pseudodifferential operators are bounded subsets of 
$\calL(H^t(\rz^n))$ for any choice of $t\in\rz$. 
\end{proof}

Let us mention that if $K=0$ in \eqref{dn04.C}, the previous proof shows that 
$f(\calA)$, $f\in H$, is a Douglis-Nirenberg system, i.e.
 $$f(\calA)=\Big(a^f_{ij}(x,D)\Big)_{1\le i,j\le q},\qquad 
   a^f_{ij}(x,\xi)\in S^{l_i-l_j}_\delta,$$
and, for suitable constants $C_k\ge 0$, 
 $$\|a^f_{ij}\|^{l_i-l_j}_{\delta,k}\le C_k\|f\|_\infty
   \qquad\forall\;f\in H\quad\forall\;k\in\nz_0.$$

\section{Systems with H\"older continuous coefficients}\label{sec:dn05}

The aim of this section is to show that Douglis-Nirenberg systems with only H\"older continuous 
coefficients (in a sense made precise below) can be treated as perturbations of smooth 
Douglis-Nirenberg systems in the sense of Section \ref{sec:dn04}. 

Let us introduce the scale of H\"older-Zygmund spaces 
 $$\calC^s_*(\rz^n):=B^s_{\infty,\infty}(\rz^n),\qquad s>0,$$
where $B^s_{p,q}(\rz^n)$ denotes the standard Besov-spaces on $\rz^n$. As before, we shall 
write also shortly $\calC^s_*$. If $s=k+r$ with $k\in\nz_0$ and $0<r<1$, then $\calC^s_*$ 
coincides with the well-known H\"older space of functions $u$ that have bounded derivatives 
$\partial^\alpha u$, $|\alpha|\le k$, and 
 $$\sup_{x\not=y}\frac{|\partial^\alpha u(x)-\partial^\alpha u(y)|}{|x-y|^r}<\infty.$$

\begin{definition}\label{dn05.1}
Let $t>0$, $\mu\in\rz$, and $0\le\delta<1$. Then 
$\calC^t_* S^\mu_{\delta}(\rz^n\times\rz^n)$ denotes the space of all functions 
$a:\rz^n\times\rz^n\to\cz$ such that, for any $\alpha\in\nz_0^n$, 
 $$|\partial^\alpha_\xi a(x,\xi)|\le C_\alpha\,\spk{\xi}^{\mu-|\alpha|}
   \qquad \forall\;x,\xi\in\rz^n$$
and 
 $$\|\partial^\alpha_\xi a(\cdot,\xi)\|_{\calC^t_*(\rz^n_x)}\le 
   C_\alpha\,\spk{\xi}^{\mu+r\delta-|\alpha|}
   \qquad \forall\;\xi\in\rz^n$$
with some constants $C_\alpha\ge 0$. Again, we write for short $\calC^t_* S^\mu_{\delta}$. 
\end{definition}

For a detailed presentation of properties of associated pseudodifferential operators we 
refer the reader to \cite{Tayl}. Let us only mention the following two results: 

\begin{proposition}\label{dn05.2}
If $a\in \calC^t_* S^\mu_{\delta}$ then, continuously, 
 $$a(x,D):H^{s+\mu}_{p}(\rz^n)\longrightarrow H^s_{p}(\rz^n)\qquad 
   \forall\;-(1-\delta)t<s<t\quad\forall\;1<p<\infty.$$
\end{proposition}

\begin{proposition}[Symbol smoothing]\label{dn05.3}
Let $a\in\calC^t_* S^\mu_{\delta}$ and $\delta<\gamma<1$. Then there exists an  
$a^\gamma\in S^\mu_{\gamma}$ such that   
 $$r^\gamma:=a-a^\gamma\in\calC^t_*S^{\mu-r(\gamma-\delta)}_{\gamma}.$$  
\end{proposition}

The previous proposition says that symbols with H\"older continuous coefficients can be 
approximated by standard smooth symbols, modulo a remainder of more negative order. 

\begin{theorem}\label{dn05.4}
Let $s\in\rz$ be fixed. Let 
$A(x,D)=\Big(a_{ij}(x,D)\Big)_{1\le i,j\le q}$ be a Douglis-Nirenberg system with 
$a_{ij}\in\calC^{t_i}_* S^{l_i+m_j}_{\delta}$ such that each $t_i$ is positive and 
\begin{equation}\label{dn05.A}
 -(1-\delta)t_i<s-l_i<t_i\qquad\forall\;1\le i\le q.
\end{equation}
Moreover, assume that $A(x,D)$ is $\Lambda$-elliptic in the sense 
of Definition {\rm\ref{dn02.4}} or {\rm\ref{dn02.5}}. Then there exists an $\alpha_0\ge 0$ 
such that 
 $$A(x,D)+\alpha:
   \mathop{\mbox{\Large$\oplus$}}_{j=1}^q H^{s+m_j}_p(\rz^n)\subset
   \mathop{\mbox{\Large$\oplus$}}_{i=1}^q H^{s-l_i}_p(\rz^n)\lra
   \mathop{\mbox{\Large$\oplus$}}_{i=1}^q H^{s-l_i}_p(\rz^n)$$
admits a bounded $H_\infty$-calculus for any $\alpha\ge\alpha_0$. 
\end{theorem}
\begin{proof}
Condition \eqref{dn05.A} together with Proposition \ref{dn05.2} ensures that $A(x,D)$ has 
the requested mapping property. By assumption on $s$ we can choose a $\gamma\in(\delta,1)$ 
such that \eqref{dn05.A} remains valid if we replace $\delta$ by $\gamma$. 
Then we choose symbols $a_{ij}^\gamma\in S^{l_i+m_j}_\gamma$ that correspond to $a_{ij}$ 
in the sense of Proposition \ref{dn05.3}. Then 
$\eps=\min\limits_{i=1}^q t_i(\gamma-\delta)$ 
is positive and it follows that 
 $$r^\gamma_{ij}(x,D):H^{s+m_j-\eps}_p(\rz^n)\to H^{s+l_i}_p(\rz^n)
   \qquad\forall\;1\le i,j\le q.$$  
Hence $K^\gamma:=\big(r^\gamma_{ij}(x,D)\big)_{1\le i,j\le q}$ is a perturbation in the sense 
of \eqref{dn04.B}. Now the result follows from Theorem \ref{dn04.5}, applied to 
$A^\gamma(x,D):=\big(a^\gamma_{ij}(x,D)\big)_{1\le i,j\le q}$, since 
$A(x,D)=A^\gamma(x,D)+K^\gamma$ by construction. 
\end{proof}

\section{The generalized thermoelastic plate equations}\label{sec:dn06}

The generalized thermoelastic plate equations on $\rz^n$ consist of the system 
\begin{align}\label{dn06.A}
 \begin{split}
 v_{tt}+L v-L^\beta w&=0\\
 cw_t+L^\alpha w+L^\beta v_t&=0
 \end{split}
\end{align}
depending on the parameters $0\le\alpha,\beta\le 1$, together with the initial conditions
 $$v(0,\cdot)=v_0,\;v_t(0,\cdot)=v_1,\;w(0,\cdot)=w_0,$$
where $L=(-\Delta)^\eta$ with some $\eta>0$. This equation has been introduced independently in 
\cite{MuRa} and \cite{AmBe}. For the special choice $\eta=2$ and $\alpha=\beta=1/2$ one obtains the 
thermoelastic plate equations 
\begin{align*}
 v_{tt}+a\Delta^2 v-b\Delta w&=0\\
 cw_t+d\Delta w+b\Delta v_t&=0. 
\end{align*}
Introducing the new variable $u=(w, v_t,L^{1/2}v)$, the system \eqref{dn06.A} can equivalently be expressed as 
\begin{equation}\label{dn06.B}
 u_t+\wt{A}(D)u=0,\qquad
 \wt{A}(\xi)=
 \begin{pmatrix}
   |\xi|^{2\alpha\eta}&|\xi|^{2\beta\eta}&0\\-|\xi|^{2\beta\eta}&0&|\xi|^{\eta}\\0&-|\xi|^{\eta}&0
 \end{pmatrix}.
\end{equation}
If we now let $\chi(\xi)$ be an arbitrary fixed 0-excision function, then $A(\xi):=\chi(\xi)\wt{A}(\xi)$ is a 
$3\times 3$-Douglis-Nirenberg system, according to the following choice of orders: 
\begin{align*}
 m_1&=2\eta(\alpha-\beta),& m_2&=0,& m_3&=2\eta\big(\mbox{$\frac{1}{2}$}+\alpha-2\beta\big),\\
 l_1&=2\beta\eta,& l_2&=2\eta(2\beta-\alpha),& l_3&=\eta.
\end{align*}
Correspondingly, we have 
\begin{equation}\label{dn06.C}
 r_1=2\eta\alpha,\qquad r_2=2\eta(2\beta-\alpha),\qquad r_3=2\eta(1+\alpha-2\beta).
\end{equation}

\begin{lemma}\label{dn06.1}
Assume that the parameters $0\le\alpha,\beta\le 1$ fulfill the conditions 
\begin{equation}\label{dn06.D}
 \alpha>\beta\quad\text{and}\quad 2\beta-\alpha>\mbox{$\frac{1}{2}$}.
\end{equation}
Then the numbers from \eqref{dn06.C} satisfy $r_1>r_2>r_3>0$ and $A(D)$ is $\Lambda$-elliptic for any 
sector $\Lambda$ which does not contain the positive half-axis. 
\end{lemma}
\begin{proof}
By direct computation we find that 
 $$\mathrm{det}\big(A[\kappa](\xi)-\lambda E_\kappa\big)=
   \begin{cases}
   |\xi|^{r_1}-\lambda&\quad:\kappa=1\\
   |\xi|^{r_1}(|\xi|^{r_2}-\lambda)&\quad:\kappa=2\\
   |\xi|^{r_1+r_2}(|\xi|^{r_3}-\lambda)&\quad:\kappa=3
   \end{cases}.$$
Now the claim follows by observing that 
 $$|s-\lambda|^2=|s|^2-2\,s\,\re\lambda+|\lambda|^2\ge\min(1,1-\cos\theta)(|s|^2+|\lambda|^2)$$ 
for any positive real $s$ and $\lambda\in\Lambda=\Lambda(\theta)$. 
\end{proof}

The assumption of strict inequalities in \eqref{dn06.D} is made to ensure the validity of \eqref{dn02.B}. 
However, as already remarked in the paragraph following \eqref{dn02.B}, we could also admit equalities 
in \eqref{dn06.D}. 

Next, we observe that $K:=\wt{A}(D)-A(D)$ is a smoothing operator in the sense that each component maps 
$H^s_p(\rz^n)$ to $H^t_p(\rz^n)$ for arbitrary $s,t\in\rz$ and $1<p<\infty$. In fact, a component of $K$ has 
the form $k(D)$ for $k(\xi)=(1-\chi)(\xi)|\xi|^\eps$ with $\eps>0$. Then the desired mapping property of 
$k(D)$ is, by composition with the order reductions $\spk{D}^s$ and $\spk{D}^t$, equivalent to the property 
that $\wt{k}(D)\in\calL(L_p(\rz^n))$ for $\wt{k}(\xi)=(1-\chi)(\xi)\spk{\xi}^{t-s}|\xi|^\eps$. However, this 
is true by the Mikhlin multiplier theorem, for example. Hence, applying Theorem \ref{dn04.5}, we obtain 
(generalizing the results of Section 3 in \cite{DeRa}): 

\begin{theorem}\label{dn06.2}
Let $\alpha,\beta$ fulfill \eqref{dn06.D}. Consider $\wt{A}(D)$ from \eqref{dn06.B} as an unbounded operator in 
 $$\calH:=H^{s-2\beta\eta}_p(\rz^n)\oplus H^{s-2\eta(2\beta-\alpha)}_p(\rz^n)\oplus H^{s-\eta}_p(\rz^n)$$
with the domain 
 $$\calD:=H^{s+2\eta(\alpha-\beta)}_p(\rz^n)\oplus H^{s}_p(\rz^n)\oplus H^{s+2\eta(1/2+\alpha-2\beta)}_p(\rz^n)$$
with some $s\in\rz$ and $1<p<\infty$. Then there exists an $\lambda_0\ge0$ such that $\wt{A}(D)+\lambda$ admits 
a bounded $H_\infty$-calculus for any $\lambda\ge\lambda_0$. In particular, if $1<q<\infty$ and $T>0$ are given, 
the equation 
 $$u_t+\wt{A}(D)u=f(t),\qquad u(0)=u_0,$$
has for each right-hand side $f\in L_q([0,T],\calH)$ and each initial value 
$u_0\in(\calD,\calH)_{\frac{1}{q},q}$\footnote{where $(\cdot,\cdot)_{\theta,q}$ refers to the real 
interpolation space} 
a unique solution belonging to $W^1_q([0,T],\calH)\,\cap\,L_q([0,T],\calD)$ 
which depends continuously on $f$ and $u_0$. 
\end{theorem}

\section{Further extensions}\label{sec:dn07}

Instead of working with the scale of Sobolev spaces $H^s_{p}(\rz^n)$ we also could have chosen to consider 
Besov-Triebel-Lizorkin spaces $B^s_{p,q}(\rz^n)$ and $F^s_{p,q}(\rz^n)$ with $1<p,q<\infty$, and also the H\"older spaces ${c}^s_*(\rz^n)$, $s>0$, which are defined as the closure of $\calC^\infty_b(\rz^n)$ 
in $\calC^s_*(\rz^n)$.\footnote{smooth functions whose derivatives of any order are bounded} 
This is due to the fact that the only property needed for the proofs is that pseudodifferential operators act continuously 
in the scale, in a sense analogous to Theorem \ref{dn02.2}. Therefore, all our results of 
Sections \ref{sec:dn03} to \ref{sec:dn06} remain valid 
in these other scales of spaces. Also they remain true for systems on compact manifolds. 

\begin{small}
\bibliographystyle{amsalpha}

\end{small}
\end{document}